\documentclass[10pt,a4]{amsart}
\usepackage[dvips]{graphicx}
\usepackage{amsmath,amssymb}
\usepackage{txfonts}
\usepackage[all]{xy}
\newtheorem{thm}{Theorem}[section]
\newtheorem{prp}[thm]{Proposition}
\newtheorem{lem}[thm]{Lemma}
\newtheorem{dfn}[thm]{Definition}
\newtheorem{cor}[thm]{Corollary}
\newtheorem{ex}{Example}
\newtheorem{rem}{Remark}
\newcommand{ \bm}[1]{\boldsymbol{#1}}

\newcommand{ \image}{\hspace{0.2em}{\rm image}\hspace{0.2em}}

\newcommand{ \map}{{\rm map}\,}
\newcommand{ \Map}{{\rm Map}\,}

\newcommand{ \SU}{S\!U}

\newcommand{ \aut}{{\rm aut}\,}

\begin{document}
\title{Finiteness of $A_n$-equivalence types of gauge groups}
\author{Mitsunobu Tsutaya}
\thanks{This paper is accepted by the Journal of the London Mathematical Society (to appear).}
\maketitle

\begin{abstract}
Let $B$ be a finite CW complex and $G$ a compact connected Lie group.
We show that the number of gauge groups of principal $G$-bundles over $B$ is finite up to $A_n$-equivalence for $n<\infty$.
As an example, we give a lower bound of the number of $A_n$-equivalence types of gauge groups of principal $\SU(2)$-bundles over $S^4$.
\end{abstract}

\section{Introduction}
Let $G$ be a topological group and $P$ be a principal $G$-bundle over a space $B$. 
A $G$-equivariant map $P\rightarrow P$ covering the identity map on $B$ is called an {\it automorphism} of $P$. 
The {\it gauge group} $\mathcal{G}(P)$ of $P$ is the topological group consisting of all automorphisms of $P$.\par
Let us consider the following problem: how many homotopy types of $\mathcal{G}(P)$ do exist for fixed $B$ and $G$? 
For $B=S^4$ and $G=\SU(2)$, Kono \cite{Kon91} has shown that there exist only six homotopy types of $\mathcal{G}(P)$.
More generally, for any finite CW complex $B$ and any compact connected Lie group $G$, Crabb and Sutherland show that that there exist only finitely many homotopy types of $\mathcal{G}(P)$ in \cite{CS00}. 
Moreover, they show that the number of homotopy types of $\mathcal{G}(P)$ as $H$-spaces, namely $H$-types of $\mathcal{G}(P)$ is also finite.
\par
Stasheff considered the concept of an $A_n$-map \cite{Sta63b} between topological monoids.
An $H$-map between topological monoids is exactly an $A_2$-map between them.
Here, we can consider the more general problem: how many $A_n${\it -equivalence types}, which we also refer as $A_n${\it -types}, of gauge groups do exist for fixed $B$ and $G$? 
Especially, let us consider this problem for any finite complex $B$ and any compact connected Lie group $G$ in case when $n$ is finite.
For $n=1$ or $2$, as stated above, Crabb and Sutherland show that it is finite.
We will give a more general result for $n\geq 3$.
\begin{thm}\label{mainthm}
Let $B$ be a finite complex and $G$ be a compact connected Lie group. 
As $P$ ranges over all principal $G$-bundles with base $B$, the number of $A_n$-equivalence types of $\mathcal{G}(P)$ is finite in case when $n$ is finite.
\end{thm}
\par
In general, the number of $A_\infty$-types of $\mathcal{G}(P)$ is not finite (\S 9).
\par
If we regard $G$ as a left $G$-space by the adjoint action $(g,x)\mapsto gxg^{-1}$, the bundle $\aut P=P\times_G G$ with fibre $G$ associated to $P$ is a fibrewise topological group, namely, a group object in the comma category $\bm{Top}\downarrow B$ and the space $\varGamma (\aut P)$ of all sections of $\aut P$ and $\mathcal{G}(P)$ are isomorphic as topological groups.
So to prove Theorem \ref{mainthm}, it is sufficient to show that the number of {\it fibrewise $A_n$-equivalence types} of $\aut P$ is finite.
\par
In \S 2, we review the terminology of fibrewise homotopy theory.
In \S 3, we review associahedra and multiplihedra. 
In \S 4, we review the definitions of {\it fibrewise $A_n$-spaces} and {\it fibrewise $A_n$-maps} and see some fundamental properties. 
These notions are the fibrewise versions of $A_n$-spaces \cite{Sta63a} and $A_n$-maps \cite{IM89} respectively. 
In \S 5, we show the classification theorem for fibrewise $A_n$-spaces with fibre $A_n$-equivalent to some fixed $A_n$-space.
In \S 6 and 7, we treat the fibrewise localization of fibrewise $A_n$-spaces.
We see that the fibrewise rationalizations of automorphism bundles are trivial. 
In \S 8, we complete the proof of Theorem \ref{mainthm}.
In \S 9, we show the existence of a counterexample to Theorem \ref{mainthm}, if we disregard the condition ``$n$ is finite'' in Theorem \ref{mainthm}.
In \S 10, we make a new observation of the gauge groups of principal $\SU(2)$-bundles over $S^4$.
Especially, we give a lower bound of the number of the $A_n$-types of such gauge groups.
\par
I am grateful to Professors Akira Kono and Daisuke Kishimoto for suggesting this problem and for many helpful discussions.

\section{Fibrewise spaces and fibrewise pointed spaces}
We follow the terminology introduced in \cite{CJ98} used in describing the fibrewise homotopy theory.
\par
Let $E$ and $B$ be spaces.
We say $E$ is a {\it fibrewise space} over $B$ if a map $\pi:E\rightarrow B$, called the {\it projection} of $E$, is given.
For each $b\in B$, we denote $E_b:=\pi^{-1}(b)$ and call $E_b$ the {\it fibre} over $b$. 
The space $B$ itself is regarded as a fibrewise space over $B$ with projection given by the identity map. 
Every space can be seen as a fibrewise space over a point with the unique projection.
\par
{\bf From now on, we always assume that all fibrewise spaces are Hurewicz fibrations.}
\par 
Let $E\stackrel{\pi}{\rightarrow}B$ and $E'\stackrel{\pi'}{\rightarrow}B$ be fibrewise spaces. 
If a map $f:E\rightarrow E'$ satisfies $\pi'f=\pi$, then $f$ is called a {\it fibrewise map} over $B$. 
We denote $I$ as the unit interval $[0,1]$. 
Regard $I\times E$ as a fibrewise space over $B$ with projection given by composing $\pi$ with the second projection $I\times E\rightarrow E$. 
Let $f,g:E\rightarrow E'$ be fibrewise maps. 
A fibrewise map $h:I\times E\rightarrow E'$ is called a {\it fibrewise homotopy} between $f$ and $g$ if $h|_{0\times E}=f$ and $h|_{1\times E}=g$. 
If there exists such $h$, then $f$ and $g$ are said to be {\it fibrewise homotopic}.
\par
For fibrewise spaces $E\stackrel{\pi}{\rightarrow}B$ and $E'\stackrel{\pi'}{\rightarrow}B$, the fibre product $E\times_BE'$ of $E$ and $E'$ is defined by
\begin{align*}
E\times_B E'=\{ \, (e,e')\in E\times E'\, |\, \pi(e)=\pi'(e') \, \} .
\end{align*}
We denote the $i$-fold fibre product of $E$ by $E^{\times_Bi}$.
\par
{\it Fibrewise mapping space} $\map_B(E,E')$ between fibrewise spaces $E$ and $E'$ over the same base $B$ is the following set with appropriate topology:
\begin{align*}
\map_B(E,E')=\coprod_{b\in B}\Map(E_b,E_b'),
\end{align*}
where $\Map (E_b,E_b')$ is the set of all continuous maps $E_b\rightarrow E_b'$.
This space is naturally a fibrewise space over $B$.
\par
{\bf In the remainder of this paper, we work in the category of fibrewise compactly-generated spaces, which is introduced in \cite{Jam95}.}
For a fibrewise space $E$ over $B$, if $E$ and $B$ are compactly-generated, then $E$ is fibrewise compactly-generated.
Let $E$, $E'$ and $E''$ be fibrewise spaces over $B$ and $f:E\times_BE'\rightarrow E''$ be a function with $\pi_{E''}f=\pi_{E\times_BE'}$, where $\pi_{E''}:E''\to B$ and $\pi_{E\times _BE'}:E\times_BE'\to B$ are the projections.
Then we define the function $f':E\rightarrow \map_B(E',E'')$ by $f'(x)(y)=f(x,y)$.
\begin{prp}[(Proposition(5.6) in \cite{Jam95})]
Assume $E$ and $E'$ are locally sliceable \cite{CJ98}.
Then $f$ is continuous if and only if $f'$ is continuous.
\end{prp}
\par 
If $E\rightarrow B$ is a Hurewicz fibration and $B$ is a CW complex, then $E$ is locally sliceable.
This criterion is sufficient for our later use.
\par
A {\it fibrewise pointed space} $E$ over $B$ is a fibrewise space $E\stackrel{\pi}{\rightarrow}B$ with a section $\sigma$ of $\pi$.
Here, each fibre $E_b$ is regarded as a pointed space with basepoint $\sigma(b)$. 
The fibrewise space $B$ over $B$ is regarded as a fibrewise pointed space with the section given by the identity map. 
Every pointed space is a fibrewise pointed space over a point.
\par
Let $B\stackrel{\sigma}{\rightarrow}E\stackrel{\pi}{\rightarrow}B$ and $B\stackrel{\sigma'}{\rightarrow}E'\stackrel{\pi'}{\rightarrow}B$ be fibrewise pointed spaces. 
A fibrewise map $f:E\rightarrow E'$ is a {\it fibrewise pointed map} if $f\sigma=\sigma'$. 
Moreover, if $f$ is a homeomorphism, we say $f$ is a {\it fibrewise pointed topological equivalence}.

\section{Review of associahedra and multiplihedra}
We will review and construct associahedra and multiplihedra using the $W$-construction in \cite{BV73} since we will use the result of Boardman and Vogt in \cite{BV73}. 
These constructions are equivalent to the original construction in \cite{Sta63a} respectively in \cite{IM89}.
This fact is also stated in \cite{For08}.
Here we remark that we do not consider the degeneracy maps of associahedra and multiplihedra.
Because we do not need to assume that an $A_n$-form has a strict unit.
Details of this will be explained in the next section.
\par
First, we recall basic definitions about trees.
\begin{dfn}
A planted plane tree (see \S 2 of \cite{Kla70}) $\tau$ is said to be an unpainted tree if each vertex in $\tau$ is not connected to exactly two edges. 
For a planted plane tree $\tau$, let $V(\tau)$ be the set of all vertices in $\tau$ and define the subset
\begin{align*}
L(\tau)=\{\,v\in V(\tau)\,|\, v \,{\rm is\, not\, the\, root\, and\, is\, connected\, to\, only\, one\, edge.}\,\}.
\end{align*}
A vertex contained in $L(\tau)$ is called a {\it leaf} and a vertex which is not the root or a leaf is called an {\it internal vertex}.
Similarly, the edge whose boundary contains the root is called the {\it root edge}, an edge whose boundary contains a leaf is called a {\it leaf edge} and an edge which is not the root edge or a leaf edge is called an {\it internal edge}.
Moreover, if each internal vertex in an unpainted tree $\tau$ is connected to just three edges, $\tau$ is said to be {\it binary}.
\end{dfn}
For a planted plane tree $\tau$, we give each edge the direction toward the root. 
Then, each internal vertex in $\tau$ has some incoming edges and the unique outgoing edge.
\begin{center}
\includegraphics[height=3cm,keepaspectratio,clip]{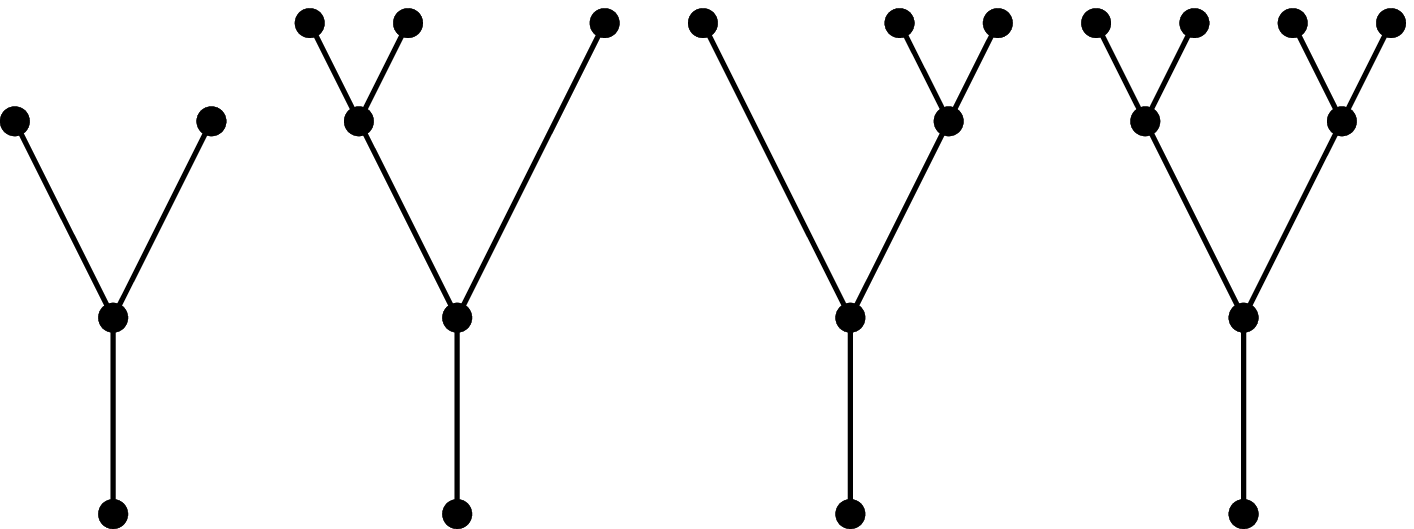}
\end{center}
\begin{rem}
For an integer $n\geq 2$, a binary unpainted tree with $n$ leaves has just $n-2$ internal edges.
\end{rem}
\begin{dfn}
Let $\tau$ be a planted plane tree.
If each internal edge of $\tau$ is labeled by an element of $I=[0,1]$, $\tau$ is said to be a {\it metric tree}.
\end{dfn}
\begin{center}
\includegraphics[height=3cm,keepaspectratio,clip]{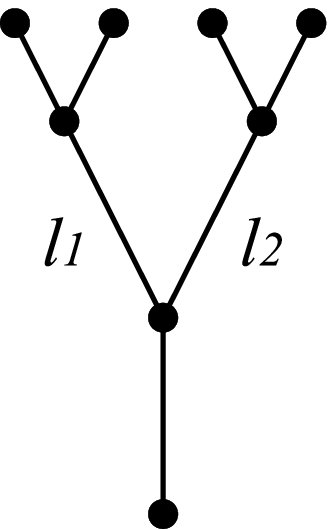}
\end{center}
\par
Denote the set of all binary unpainted trees with $n$ leaves by $T_n$.
The space $T\mathcal{A}_n$ consisting of all binary unpainted metric trees with $n$ leaves is topologized by $T\mathcal{A}_n=T_n\times I^{n-2}$.
Now we define an equivalence relation in $T\mathcal{A}_n$.
For $\rho ,\tau \in T\mathcal{A}_n$, remove all internal edges with length 0 and unite the end vertices of each removed edges.
\begin{center}
\includegraphics[height=3cm,keepaspectratio,clip]{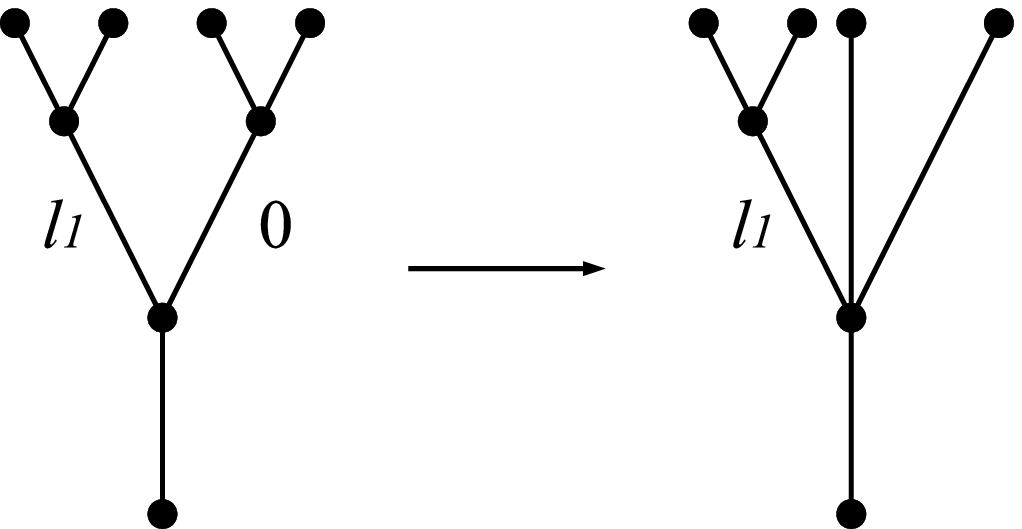}
\end{center}
Then we have the new (in general, not binary) unpainted metric trees $R(\rho)$, $R(\tau)$.
We say $\rho$ and $\tau$ are equivalent if $R(\rho)=R(\tau)$.
This relation defines the quotient space $\mathcal{K}_n$ of $T\mathcal{A}_n$.
For example, $\mathcal{K}_2$ is a point, $\mathcal{K}_3$ is a line segment and $\mathcal{K}_4$ is a pentagon.
\par
Let us define the grafting map.
For $\rho \in T\mathcal{A}_r$, $\tau \in T\mathcal{A}_t$ and an integer $1\leq k\leq r$, we can make the new binary unpainted metric tree $\partial_k(r,t)(\rho, \tau)$ by identifying the root edge of $\tau$ and the $k$-th leaf edge of $\rho$, where this identification is compatible with the direction of edges and the length of the new internal edge is $1\in I$.
\begin{center}
\includegraphics[height=4cm,keepaspectratio,clip]{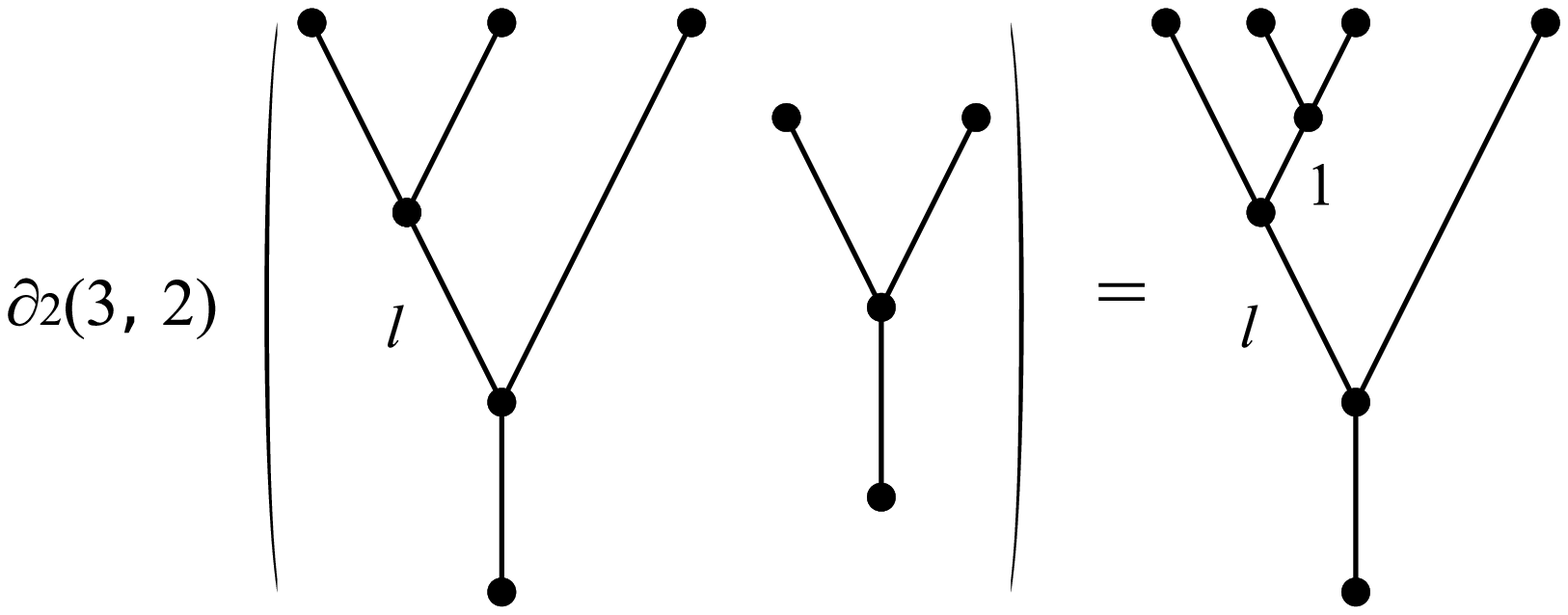}
\end{center}
This defines a continuous map $\partial_k(r,s):\mathcal{K}_r\times \mathcal{K}_t\rightarrow \mathcal{K}_{r+t-1}$.
These grafting maps satisfy the following conditions:
\begin{align*}
&\partial_j (p,r+t-1)(1\times \partial_k(r,t))=\partial_{j+k-1}(p+r-1,t)(\partial_j(p,r)\times 1),\\
&\partial_{j+r-1}(p+r-1,t)(\partial_k(p,r)\times 1)=\partial_k(p+t-1,r)(\partial_j(q,t)\times 1)(1\times T)\tag*{for $k<j$,}
\end{align*}
where $T$ transposes the factors.
Let $\mathcal{L}_n$ be the union of images of these grafting maps in $\mathcal{K}_n$.
Then $\mathcal{K}_n$ is homeomorphic to the cone $C\mathcal{L}_n$ of $\mathcal{L}_n$.
Therefore, these spaces are Stasheff's associahedra in \cite{Sta63a}.
This implies that there is a homeomorphism $(\mathcal{K}_n,\mathcal{L}_n)\simeq (D^{n-2},S^{n-3})$ for $n\geq 3$.
\par
In the above construction, if we replace unpainted metric trees by {\it painted metric trees}, then we can construct multiplihedra.
We consider two types of edges of painted trees, say {\it unpainted edges} and {\it painted edges}.
\begin{dfn}
A {\it painted tree} is a planted plane tree with edges labeled by the set $\{\,{\rm painted,\, unpainted}\, \}$ satisfying the following conditions:\\
(i) each of its internal vertices is one of the following types:\\
 Type I: all incoming edges and the outgoing edge are unpainted,\\
 Type II: all incoming edges and the outgoing edge are painted,\\
 Type III: all incoming edges are unpainted and the outgoing edge is painted,\\
where the number of incoming edges of a vertex of type I or II is greater than 1,\\
(ii) all leaf edges are unpainted while the root edge is painted.
\begin{center}
\includegraphics[height=1.6cm,keepaspectratio,clip]{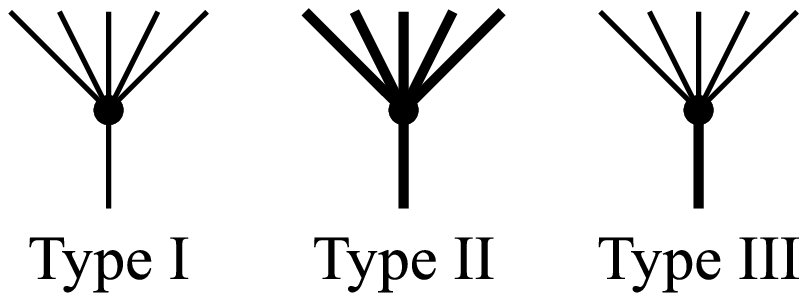}
\end{center}
\par
A painted tree is called {\it binary} if the number of incoming edges of each internal vertex of type I or II is 2 and the number of incoming edges of each internal vertex of type III is 1.
\begin{center}
\includegraphics[height=1.6cm,keepaspectratio,clip]{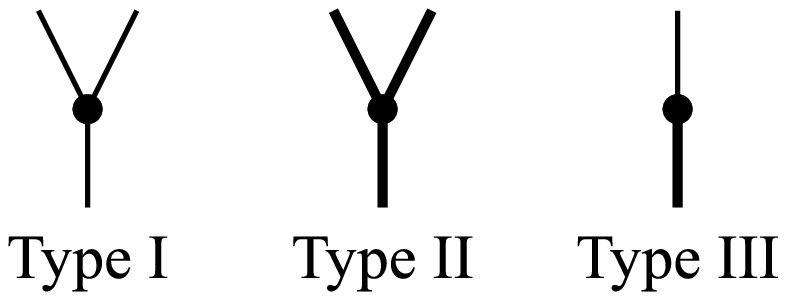}
\end{center}
\end{dfn}
\par
For painted metric trees, we consider the following two types of grafting maps.
\par
Like unpainted metric trees, we can graft an unpainted metric tree onto a painted metric tree.
For a painted metric tree $\rho$ with $r$ leaves, an unpainted metric tree $\tau$ with $t$ leaves and an integer $1\leq k\leq r$, we can make the new painted metric tree $\delta_k(r,t)(\rho, \tau)$ by identifying the root edge of $\tau$ and the $k$-th leaf edge of $\rho$, where this identification is compatible with the direction of edges, the length of the new internal edge is $1\in I$ and each edge in $\delta_k(r,t)(\rho, \tau)$ which comes from $\tau$ is unpainted.
\par
Conversely, we can graft painted metric trees onto an unpainted metric tree.
For an unpainted metric tree $\tau$ with $t$ leaves and painted metric trees $\rho_1,\cdots ,\rho_t$ such that each $\rho_i$ has $r_i$ leaves, $\delta(t,r_1,\cdots ,r_t)(\tau ,\rho_1,\cdots ,\rho_t)$ is the painted metric tree constructed by identifying the root edge of $\rho_i$ and the $i$-th leaf edge of $\tau$ for each $i$, where this identification is compatible with the direction of edges, the lengths of the new internal edges are $1\in I$ and each edge which comes from $\tau$ is painted.
\begin{center}
\includegraphics[height=3cm,keepaspectratio,clip]{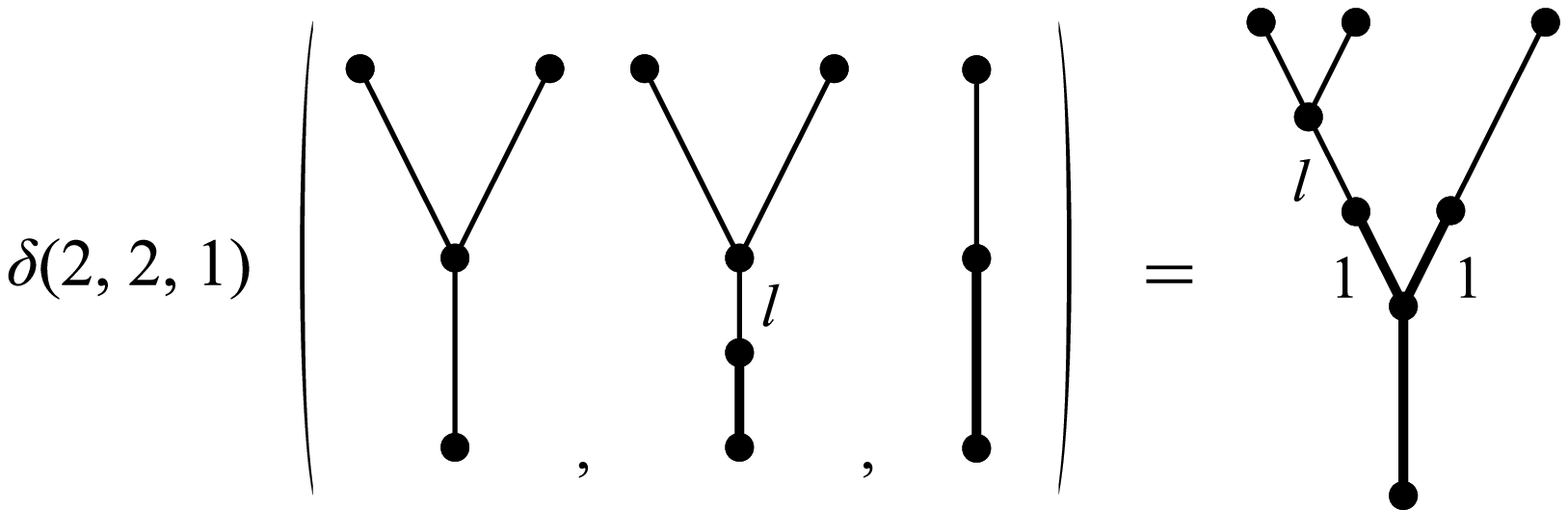}
\end{center}
\par
It is not quite appropriate for our purpose to consider all painted metric trees, so we consider a new class of painted metric trees.
For an painted metric tree $\tau$, let $M(\tau)$ be the length of the longest internal edge in $\tau$.
If $\tau$ has no internal vertices, then define $M(\tau)=0$.
\begin{dfn}
{\it Level-trees} are painted metric trees inductively defined as follows:\\
(I) If a painted metric tree $\tau$ satisfies the condition $M(\tau)=0$, then $\tau$ is a level-tree.\\
(II) Let $\tau$ be a painted metric tree with $n$ internal edges such that $M(\tau)>0$.
Define a new painted metric tree $\tilde \tau$ as follows:\\
(i) $\tilde \tau =\tau$ as painted trees,\\
(ii) for each internal edge $e\subset \tilde \tau$, if the length of $e$ in $\tau$ is $\ell \in I$, then the length of $e$ in $\tilde \tau$ is $\ell /M(\tau) \in I$.\\
If $\tilde \tau =\delta_k(r,s)(\rho,\sigma)$ for some level-tree $\rho$ and unpainted metric tree $\sigma$ or $\tilde \tau =\delta(r,s_1,\cdots ,s_r)(\rho' ,\sigma'_1,\cdots ,\sigma'_r)$ for some unpainted metric tree $\rho$ and level-trees $\sigma_1,\cdots ,\sigma_r$, then $\tau$ is a level-tree.
Here we remark that the number of internal edges in each of such level-trees $\rho,\sigma'_1,\cdots ,\sigma'_r$ is less than $n$ and can define level-trees inductively.
\end{dfn}
\begin{ex}
Consider painted metric trees $\tau_1,\tau_2,\tau_3$ as the following figure.
\begin{center}
\includegraphics[height=3.5cm,keepaspectratio,clip]{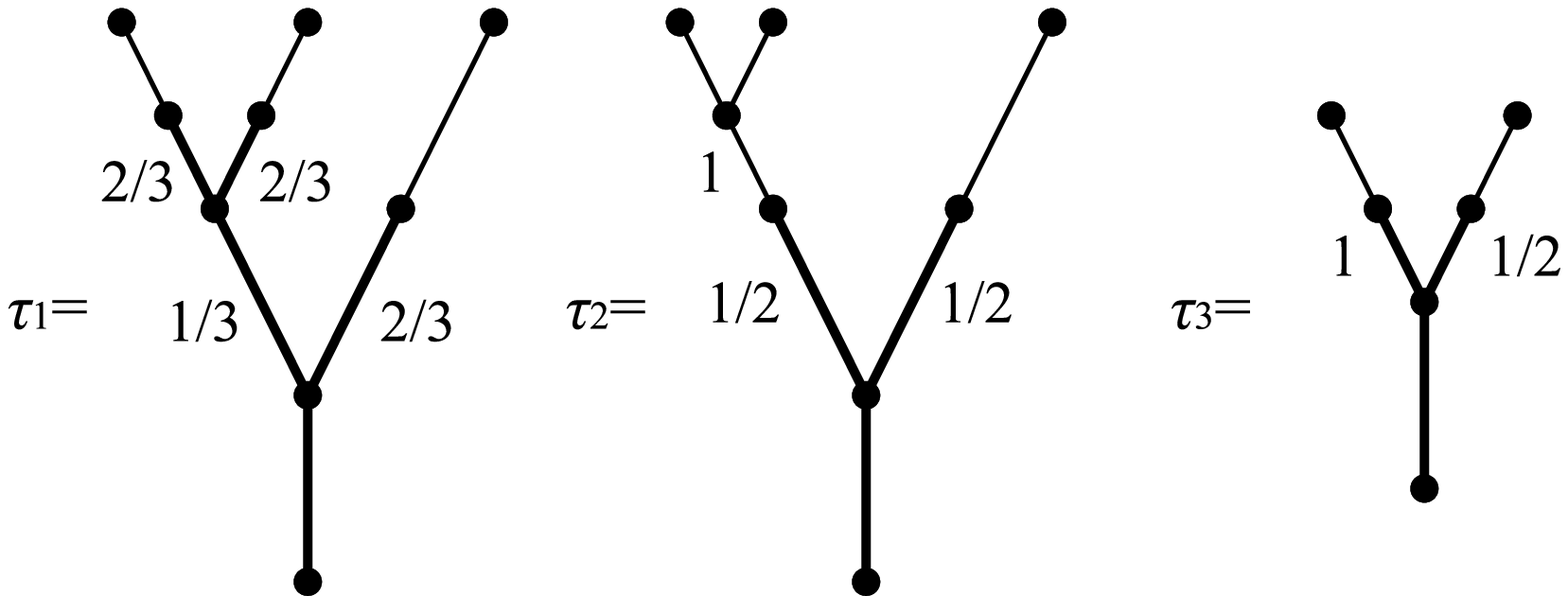}
\end{center}
(i)
Since $M(\tau_1)=2/3$, $\tilde \tau_1$ is described as the following figure.
\begin{center}
\includegraphics[height=3.5cm,keepaspectratio,clip]{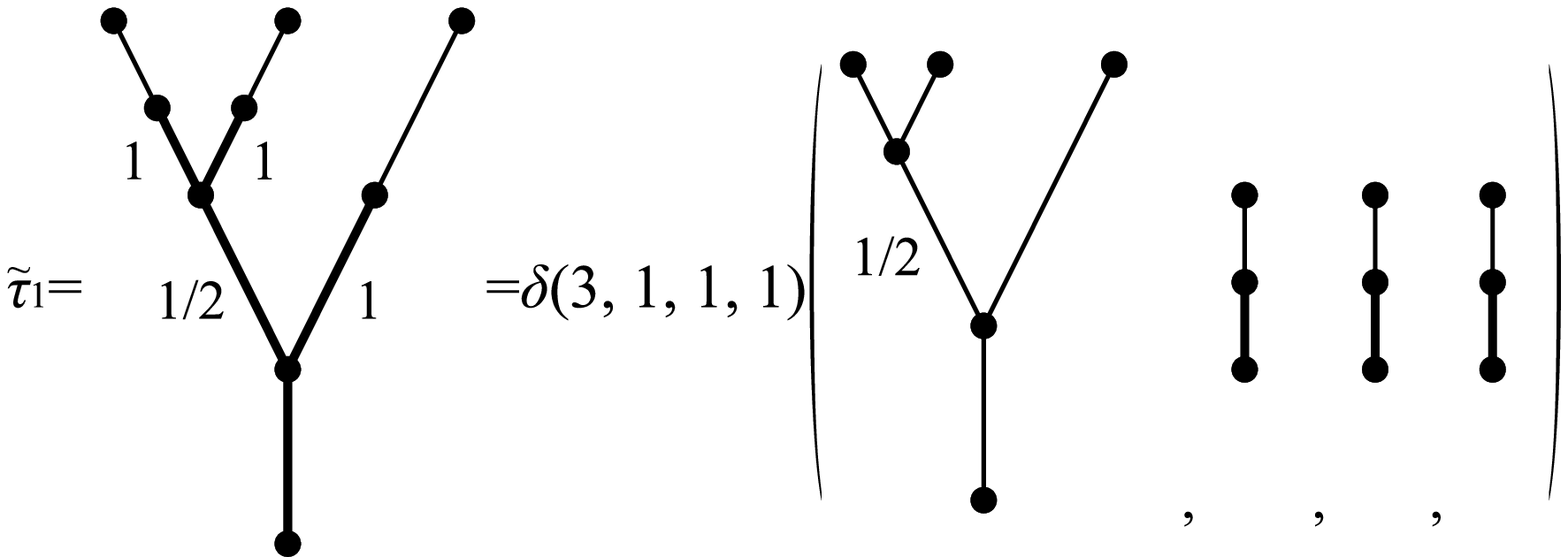}
\end{center}
Hence $\tau_1$ is a level-tree.\\
(ii)
Since $M(\tau_2)=1$, $\tilde \tau_2=\tau_2$ and $\tilde \tau_2$ is decomposed as follows.
\begin{center}
\includegraphics[height=3.5cm,keepaspectratio,clip]{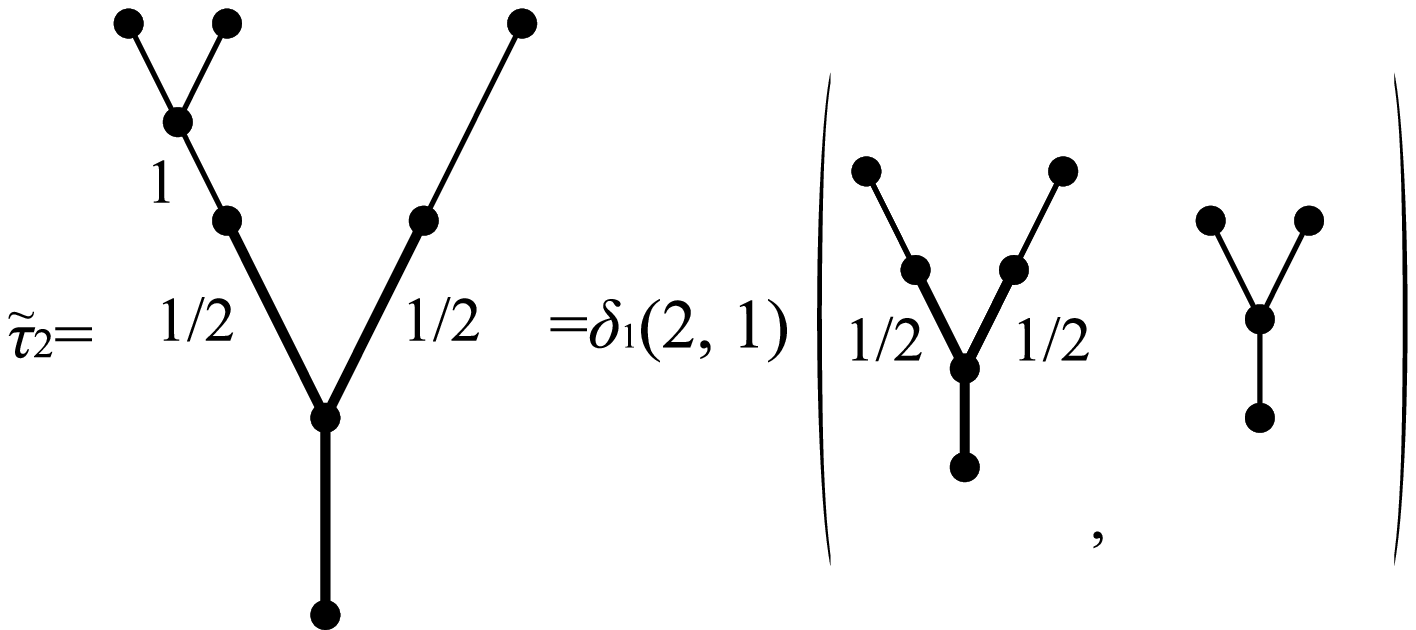}
\end{center}
As easily checked, the painted metric tree in the right hand side is a level-tree.
Then $\tau_2$ is also a level-tree.\\
(iii)
Since $M(\tau_3)=1$, $\tilde \tau_3=\tau_3$. 
The painted metric tree $\tilde \tau_3$, however, can not be described as any grafting of some painted metric trees and unpainted metric trees.
Thus $\tau_3$ is not a level-tree.
\end{ex}
\par
Let $T\mathcal{M}_n$ be the set consisting of all binary level-trees with $n$-leaves.
We topologize $T\mathcal{M}_n$ by the embedding
\begin{align*}
T\mathcal{M}_n\subset \bigcup _{\substack{\tau :\,{\rm a\, binary\, painted\, tree}\\ {\rm with\,}n\,{\rm leaves}}}\{\tau\}\times I^{\times N(\tau)},
\end{align*}
where $N(\tau)$ denotes the number of internal edges in $\tau$.
Just like unpainted metric trees, we consider reduction of edges with length 0.
For $\tau \in T\mathcal{M}_n$, remove all internal edges with length 0 and unite the end vertices of each removed edges.
We denote this new level-tree by $R(\tau)$.
We say $\rho$ and $\tau \in T\mathcal{M}_n$ are equivalent if $R(\rho)=R(\tau)$ as painted metric trees.
Let $\mathcal{J}_n$ be the quotient space of $T\mathcal{M}_n$ by this relation.
For example, $\mathcal{J}_1$ is a point, $\mathcal{J}_2$ is a line segment and $\mathcal{J}_3$ is a hexagon.
\par 
The grafting constructions define the continuous maps $\delta_k(r,t):\mathcal{J}_r\times \mathcal{K}_t\rightarrow \mathcal{J}_{r+t-1}$ and $\delta(t,r_1,\cdots ,r_t):\mathcal{K}_t\times \mathcal{J}_{r_1}\times \cdots \mathcal{J}_{r_t}\rightarrow \mathcal{J}_{r_1+\cdots +r_t}$.
These maps satisfy the following conditions:
\begin{align*}
&\delta_j(p,r+t-1)(1\times \partial_k(r,t))=\delta_{j+k-1}(p+r-1,t)(\delta_j(p,r)\times 1),\\
&\delta_{j+r-1}(p+r-1,t)(\delta_k(p,r)\times 1)=\delta_k(p+t-1,r)(\delta_j(p,t)\times 1)(1\times T)\tag*{for $k<j$,}\\
&\delta_{p_1+\cdots +p_{j-1}+k}(p,r)(\delta(t,p_1,\cdots,p_t)\times 1)\\
=&\delta(t,p_1,\cdots ,p_{j-1},p_j+r-1,p_{j+1},\cdots,p_t)(1^{\times j}\times \delta_k(p_j,r)\times 1^{\times q-j})T_j\tag*{for $p_1+\cdots +p_t=p$ and $1\leq k\leq r_j$,}\\
&\delta(r+t-1,p_1,\cdots,p_{r+t-1})(\partial_k(r,t)\times 1^{\times r+t-1})\\
=&\delta(r,p_1,\cdots,p_{k-1},p_k+\cdots +p_{k+t-1},p_{k+t},\cdots,p_{r+t-1})(1^{\times k}\times \delta(t,p_k,\cdots,p_{k+t-1})\times 1^{\times r-k})T'_k,
\end{align*}
where the maps $T_j$ and $T'_k$ are defined by $T_j(\tau,\pi_1,\! \cdots \! ,\pi_t,\rho)=(\tau,\pi_1,\! \cdots \! ,\pi_j,\rho,\pi_{j+1},\! \cdots \! ,\pi_t)$ and $T'_k(\rho,\tau,\pi_1,\! \cdots \! ,\pi_{r+t-1})=(\rho,\pi_1,\! \cdots \! ,\pi_{k-1},\tau,\pi_k,\! \cdots \! ,\pi_{r+t-1})$.
Let $\mathcal{H}_n$ be the union of images of these grafting maps in $\mathcal{J}_n$.
Then $\mathcal{J}_n$ is homeomorphic to the cone $C\mathcal{H}_n$ of $\mathcal{H}_n$.
Therefore, $\mathcal{J}_n$ is the $n$-th multiplihedron constructed by Iwase and Mimura in \cite{IM89}.
This implies that there is a homeomorphism $(\mathcal{J}_n,\mathcal{H}_n)\simeq (D^{n-1},S^{n-2})$ for $n\geq 2$.
\par
We compare the above constructions of associahedra and multiplihedra with the $W$-construction in \cite{BV73}.
Let $\mathcal{A}$ be the PRO of semigroups and $\mathcal{L}_1$ be the linear category $0\rightarrow 1$.
In other words, $\mathcal{A}$ is the PRO such that $\mathcal{A}(0,1)$ is empty and $\mathcal{A}(n,1)$ is a point for $n\geq 1$.
From the above construction, $W\mathcal{A}(1,1)$ is a point, $W\mathcal{A}(n,1)=\mathcal{K}_n$ for $n\geq 2$, $LW(\mathcal{A}\otimes \mathcal{L}_1)(n^0,1^1)=\mathcal{J}_n$ for $n\geq 1$ and the compositions in these PROs are compatible with our grafting maps.

\section{Fibrewise $A_n$-spaces and fibrewise $A_n$-maps}
We shall review fibrewise $A_n$-spaces and fibrewise $A_n$-maps. 
For any fibrewise space $E$ over $B$, we will consider $\mathcal{K}_i\times E$ and $\mathcal{J}_i\times E$ as a fibrewise spaces over $B$.
\begin{dfn}
A fibrewise space $E$ over $B$ is called a {\it fibrewise $A_n$-space} (without unit) if a family of fibrewise maps $\{m_i:\mathcal{K}_i\times E^{\times_Bi}\rightarrow E\}_{i=2}^n$, called a {\it fibrewise $A_n$-form} of $E$, is given and satisfies the following condition:
\begin{align*}
&m_i(\partial_k(r,t)(\rho,\tau);x_1,\cdots,x_i)=m_r(\rho;x_1,\cdots,x_{k-1},m_t(\tau;x_k,\cdots,x_{k+t-1}),x_{k+t},\cdots,x_i)\tag*{for $x_l\in E_b,\rho \in \mathcal{K}_r,\tau \in \mathcal{K}_t$.}
\end{align*}\par
Let $(E,\{m_i\})$ and $(E',\{m_i'\})$ be fibrewise $A_n$-spaces over $B$. 
A fibrewise map $f:E\rightarrow E'$ is called a {\it fibrewise $A_n$-map} if there exists a family of fibrewise maps $\{f_i:\mathcal{J}_i\times E^{\times_Bi}\rightarrow E'\}_{i=1}^n$, called a {\it fibrewise $A_n$-form} of $f$, such that
\begin{align*}
&f_1=f\\
&f_i(\delta_k(r,t)(\rho,\tau);x_1,\cdots,x_i)=f_r(\rho;x_1,\cdots,x_{k-1},m_t(\tau;x_k,\cdots,x_{k+t-1}),x_{k+t},\cdots,x_i)\tag*{for $x_l\in E_b,\rho \in \mathcal{J}_r,\tau \in \mathcal{K}_t$,}\\
&f_i(\delta(t,r_1,\cdots,r_t)(\tau,\rho_1,\cdots,\rho_t);x_1,\cdots,x_i)=m_t'(\tau;f_{r_1}(\rho;x_1,\cdots,x_{r_1}),\cdots,f_{r_t}(\rho_t;x_{r_1+\cdots+r_{t-1}+1},\cdots,x_i))\tag*{for $x_l\in E_b,\rho_s \in \mathcal{J}_{r_s},\tau \in \mathcal{K}_t$.}
\end{align*}
A {\it fibrewise $A_n$-equivalence} is a fibrewise $A_n$-map which is also a fibrewise homotopy equivalence. 
If there exists a fibrewise $A_n$-equivalence between two fibrewise $A_n$-spaces, then they are said to be {\it fibrewise $A_n$-equivalent}.\par
In particular, a fibrewise $A_n$-space over a point is called an {\it $A_n$-space}.
The terms such as {\it $A_n$-map} etc. are similarly defined.
\end{dfn}
\begin{rem}
(i) We do not suppose the existence of strict or homotopy units of fibrewise $A_n$-spaces.
Because the author does not know whether the universal fibrewise $A_n$-space $E_n(G)$ constructed in \S 5 has a fibrewise homotopy unit or not.
Every fibrewise pointed space is a fibrewise $A_\infty$-space by a trivial fibrewise $A_\infty$-form in the above sense.
But we do not have any difficulties in showing Theorem \ref{mainthm} with our definition of fibrewise $A_n$-forms.
Because our definition of $A_n$-maps between topological monoids agrees with the definition in \cite{Sta63b}.
\\
(ii) We do not require that fibrewise $A_n$-spaces are fibrewise pointed and fibrewise $A_n$-forms preserve the basepoint in each fibre.
But even if we do, the following arguments are verified analogously.
In that situation, some spaces are replaced by ``based'' ones.
For example, the universal fibration is replaced by the fibrewise pointed universal fibration and the mapping space $\Map(G^n,G)$ in \S 5 is replaced by $\Map^*(G^n,G)$, which is the subspace of $\Map(G^n,G)$ consisting of maps $G^n\rightarrow G$ preserving the basepoint.
\\
(iii) Fibrewise $A_n$-spaces are exactly multiplicative functors from $Q^nW\mathcal{A}$ (see Remark 3.19 in \cite{BV73}) to the category of fibrewise spaces in the sense of \cite{BV73}.
Similarly, fibrewise $A_n$-maps are fibrewise {\it level-tree maps between fibrewise $Q^nW\mathcal{A}$-spaces}.
Here, we define $Q^nLW(\mathcal{A}\otimes \mathcal{L}_1)$ as the PRO-subcategory of $LW(\mathcal{A}\otimes \mathcal{L}_1)$ generated by the morphisms of $LW(\mathcal{A}\otimes \mathcal{L}_1)(\underline{i},k)$, $\underline{i}:[r]\rightarrow \rm{ob}\,\mathcal{L}_1$ with $r\leq n$ and a fibrewise level-tree maps between fibrewise $Q^nW\mathcal{A}$-spaces $E$ and $E'$ means a multiplicative functor from $Q^nLW(\mathcal{A}\otimes \mathcal{L}_1)$ to the category of fibrewise spaces of which the restrictions to $d^1Q^nW\mathcal{A}$ and $d^0Q^nW\mathcal{A}$ give fibrewise $Q^nW\mathcal{A}$-spaces $E$ and $E'$ respectively.
\end{rem}
\begin{ex}\label{example1}
(i) Every topological monoid is an $A_\infty$-space.
More generally, every fibrewise topological monoid is a fibrewise $A_\infty$-space.
Here, a fibrewise pointed space $B\stackrel{\sigma}{\rightarrow}E\stackrel{\pi}{\rightarrow}B$ ($\sigma$ is a section of $\pi$) with a fibrewise map $m:E\times_BE\rightarrow E$ is a {\it fibrewise topological monoid} if $m(1\times_Bm)=m(m\times_B 1)$ and $m(\sigma(\pi(x)),x)=m(x,\sigma(\pi (x)))=x$ for each $x\in E$. 
If $(E,m)$ is a fibrewise topological monoid, then the family of maps $\{m_i:K_i\times E^{\times_B i}\rightarrow E\}$ defined by $m_i(\rho ;x_1,\cdots ,x_i)=m(x_1,\cdots ,m(x_{i-1},x_i)\cdots )$
is a fibrewise $A_\infty$-form of $E$.\\
(ii) An $H$-space is an $A_2$-space. A homotopy associative $H$-space $(G,m)$ together with its associating homotopy $m(1\times m)\simeq m(m\times 1)$ is an $A_3$-space.\\
(iii) For a pointed space $X$, the based loop space $\varOmega X$ of $X$ is naturally an $A_\infty$-space.
In fact, an $A_\infty$-form of $\varOmega X$ is constructed as follows.
Let 
\begin{align*}
P_i=\{ \, (t_0,t_1,\cdots ,t_i)\in I^{i+1}\,|\, 0=t_0<t_1<\cdots <t_i=1\, \}
\end{align*}
be the space of partitions of the unit interval $I=[0,1]$.
Since $P_i$ is contractible, one can construct an $A_\infty$-form $\{m_i:\mathcal{K}_i\times (\varOmega X)^{\times i}\to \varOmega X\}$ of $\varOmega X$ such that
\begin{align*}
m_i(\rho;\ell _1,\cdots ,\ell _i)(t)=\ell _k\left(\frac{t-\omega^i_{k-1}(\rho)}{\omega^i_k(\rho)-\omega^i_{k-1}(\rho)}\right)
\end{align*}
for $\rho \in \mathcal{K}_i$, $\ell_1,\cdots ,\ell_i \in \varOmega X$, $\omega^i_{k-1}(\rho)\leq t\leq \omega^i_k$, where $\omega^i=(\omega^i_0,\omega^i_1,\cdots ,\omega^i_i):\mathcal{K}_i\to P_i$.
One can take this family $\{\omega^i\}$ as independent of $X$.
Since $P_i$ is contractible, such $A_\infty$-form of $\varOmega X$ is unique up to homotopy of $A_\infty$-forms.
\\
(iv) Every homomorphism between topological monoids is an $A_\infty$-map.
We can consider a more general situation.
Let $(E,\{m_i\})$ and $(E',\{m_i'\})$ be fibrewise $A_n$-spaces.
A fibrewise map $f:E\rightarrow E'$ is called a {\it fibrewise $A_n$-homomorphism} if $m_i'(1\times f^{\times_Bi})=fm_i$ for each $i$, where a fibrewise $A_n$-form of $f$ is constructed by forgetting the painting of trees in $\mathcal{J}_r$.
For example, if $X$ and $Y$ are pointed spaces and $f:X\rightarrow Y$ is a pointed map, then the map $\varOmega f:\varOmega X\rightarrow \varOmega Y$ is an $A_\infty$-homomorphism.
\\
(v) An $H$-map between $H$-spaces is an $A_2$-map.
\end{ex}
\par 
The results by Boardman and Vogt in Chapter IV, \S 2 and \S 3 of \cite{BV73} remain true even if $W\mathcal{B}$-spaces are replaced by ``fibrewise $W\mathcal{B}$-spaces''.
Hence, we have the following propositions.
\begin{prp}[(Corollary 4.15 in \cite{BV73})]\label{composition}
Let $E$, $E'$ and $E''$ be fibrewise $A_n$-spaces over $B$ and $f:E\rightarrow E'$ and $g:E'\rightarrow E''$ be fibrewise $A_n$-maps. 
Then $gf:E\rightarrow E''$ is also a fibrewise $A_n$-map.
\end{prp}
\begin{prp}[(Corollary 4.20 in \cite{BV73})]\label{equivalence1}
Let $E$ be a fibrewise space over $B$ and $E'$ be a fibrewise $A_n$-space over $B$. 
If $f:E\rightarrow E'$ is a fibrewise homotopy equivalence, then there exists a fibrewise $A_n$-form of $E$ such that $f$ is a fibrewise $A_n$-equivalence. 
\end{prp}
\begin{prp}[(Corollary 4.21 in \cite{BV73})]\label{inverse}
Let $E$ and $E'$ be fibrewise $A_n$-spaces over $B$ and $f:E\rightarrow E'$ be a fibrewise $A_n$-equivalence.
Then the fibrewise homotopy inverse $g:E'\rightarrow E$ of $f$ is also a fibrewise $A_n$-equivalence.
\end{prp}
\begin{rem}
(i) From Propositions \ref{composition} and \ref{inverse}, the fibrewise $A_n$-equivalence is an equivalence relation.
\\
(ii) Boardman and Vogt have not explicitly shown the level-tree map version.
However, Proposition 4.6 in \cite{BV73} guarantees us the above propositions.
The $Q^nW\mathcal{B}$-space version can also be verified.
\\ 
(iii) For the fibrewise pointed version, we need to assume the fibrewise homotopy extension property of the sections of fibrewise pointed spaces.
This property corresponds to well-pointedness of pointed spaces.
For the pointed version, see Chapter 5, \S 5 in \cite{BV73}.
\end{rem}

\section{The classification theorem for fibrewise $A_n$-spaces}
Let $B'$ be a space, $f:B'\rightarrow B$ be a map and $E$ be a fibrewise $A_n$-space over $B$. 
The pull-back $f^*E$ of $E$ by $f$ is naturally a fibrewise $A_n$-space.
\par
Let $E$ and $E'$ be fibrewise $A_n$-spaces over $B$ and $B'$ respectively.
We say a fibre map $f$ covering $\bar f:B\rightarrow B'$ is a {\it fibrewise $A_n$-map} over $\bar f$ if the induced map $E\rightarrow \bar f^*E'$ by $f$ is a fibrewise $A_n$-map.
In particular, $\bar f^*E'\rightarrow E'$ is a fibrewise $A_n$-map over $\bar f$.
\begin{prp}\label{homotopy fibrewise}
Let $E$ and $E'$ be fibrewise $A_n$-spaces over $B$ and $B'$ respectively.
If a fibrewise $A_n$-map $f:E\rightarrow E'$ over $\bar f$ is homotopic to a fibre map $g:E\rightarrow E'$ covering $\bar g:B\rightarrow B'$ with a homotopy from $f$ to $g$ covering a homotopy from $\bar f$ to $\bar g$, then $g$ is a fibrewise $A_n$-map over $\bar g$.
\end{prp}
\begin{proof}
Take a homotopy $h:I\times E\rightarrow E'$ between $f$ and $g$ covering a homotopy $\bar h:I\times B\rightarrow B'$ between $\bar f$ and $\bar g$.
We construct a fibrewise $A_n$-form $\{h_i':\mathcal{J}_i\times I\times E^{\times_Bi}\rightarrow \bar h^*E'\}$ of the fibrewise map $h':I\times E\rightarrow \bar h^*E'$ induced by $h$, where $\{h_i'|_{0\times B}\}$ is equal to the fibrewise $A_n$-form $\{f_i'\}$ of the fibrewise map $E\rightarrow \bar f^*E'$ induced by $f$. \par
Assume we have constructed $h_1',\cdots ,h_{i-1}'$.
Since $E'$ has the homotopy lifting property and $(\mathcal{J}_i,\mathcal{H}_i)\simeq (D^{i-1},S^{i-2})$, we can extend the fibrewise map $h_i':((\mathcal{J}_i\times 0)\cup (\mathcal{H}_i\times I))\times E^{\times_Bi}\rightarrow \bar h^*E'$ defined by
\begin{align*}
&h_i'(\rho ,s;x_1,\! \cdots \!,x_i)\\
&=\! \left\{
\begin{array}{ll}
f_i'(\rho ;x_1,\cdots ,x_i) & (s=0) \\
h_r'(\rho' ,s;x_1,\cdots ,x_{k-1},m_t(\tau' ;x_k,\cdots ,x_{k+t-1}),x_{k+t},\cdots ,x_i) & (\rho =\delta _k(r,t)(\rho' ,\tau' )) \\
m_t'(\tau' ;h_{r_1}'(\rho _1',s;x_1,\! \cdots \! ,x_{r_1}),\! \cdots \! ,h_{r_t}'(\rho _t',s;x_{r_1+\cdots +r_{t-1}+1},\! \cdots \! ,x_i)) & (\rho \! =\! \delta (t,r_1,\! \cdots \! ,r_t)(\tau' \! ,\rho_1' ,\! \cdots \! ,\rho_t' )) \\
\end{array}
\right.
\end{align*}
to $\mathcal{J}_i\times I\times E^{\times_Bi}$, where $\{m_i\}$ and $\{m_i'\}$ are fibrewise $A_n$-forms of $E$ and $E'$, respectively.
We have constructed a desired map $h_i$.
Therefore $g$ is a fibrewise $A_n$-map over $\bar g$.
\end{proof}
\begin{rem}
In the pointed case, since we have used the homotopy lifting property, we have to assume that $E'$ is an {\it ex-fibration} (see \cite{CJ98}).
\end{rem}
\begin{prp}\label{class1}
Let $E$ be a fibrewise $A_n$-space over $B$ and $B'$ be a space.
If maps $f,g:B'\rightarrow B$ are homotopic, then $f^*E$ and $g^*E$ are fibrewise $A_n$-equivalent.
\end{prp}
\begin{proof}
It is sufficient to prove that for a fibrewise $A_n$-space $E$ over $I\times B$, $E|_{0\times B}$ and $E|_{1\times B}$ are fibrewise $A_n$-equivalent.
From the homotopy lifting property of $E$, there exists a fibrewise map $h:I\times (E|_{0\times B})\rightarrow E$ over $I\times B$.
From Proposition \ref{homotopy fibrewise}, $h|_{1\times E}:E|_{0\times B}\rightarrow E|_{1\times E}$ is a fibrewise $A_n$-map.
Since $h|_{1\times E}$ is a fibrewise homotopy equivalence, $E|_{0\times B}$ and $E|_{1\times B}$ are fibrewise $A_n$-equivalent.
\end{proof}
\par
In the rest of this section, we will construct the classifying space for fibrewise $A_n$-spaces with fibres $A_n$-equivalent to a fixed $A_n$-space by the same method as in \cite{CS00}. 
Let $(G,\{\mu_i\})$ be an $A_n$-space and $B$ be a space, where both $G$ and $B$ has homotopy types of CW complexes. 
Assume every fibre of a fibrewise $A_n$-space is $A_n$-equivalent to $G$ in the rest of this section.\par
Let $E$ be a fibrewise space such that fibres of $E$ are homotopy equivalent to $G$. Define
\begin{align*}
M_n[E]=\coprod_{b\in B}\left. \left\{\, \{m_i\} \in \prod_{i=2}^n \Map(\mathcal{K}_i\times E_b^{\times i},E_b)\, \right| 
\begin{array}{l}
\{m_i\}:{\rm an\,}A_n{\rm \mathchar`-form\, of\,}E_b{\rm \, such\, that\,} \\
(E_b,\{m_i\}){\rm \,and\,}(G,\{\mu_i\}){\rm \,are\,}A_n{\rm \mathchar`-equivalent} \\
\end{array}
\right\}.
\end{align*}
\par
Here, we have the projection $M_n[E]\rightarrow M_{n-1}[E]$ defined by forgetting $m_n$. 
Since the inclusion $\mathcal{L}_n\rightarrow \mathcal{K}_n$ has the homotopy extension property, this projection is a Hurewicz fibration, of which the fibres are homotopy equivalent to the component $\varOmega_0^{n-2}\Map(G^{\times n},G)$ containing the constant map $S^{n-2}\to \Map(G^{\times n},G)$ of $\varOmega^{n-2}\Map(G^{\times n},G)$ with basepoint of $\Map(G^{\times n},G)$ given by the map $(x_1,\cdots,x_n)\mapsto \mu_2(x_1,\cdots,\mu_2(x_{n-1},x_n)\cdots)$.
Then a fibrewise space $E_n[E]$ is defined as the pull-back of $E$ by the projection $M_n[E]\rightarrow B$. 
Define the fibrewise map $\tilde m_i:K_i\times E_n[E]^{\times_{M_n[E]}i}\rightarrow E_n[E]$ by $\tilde m_i(\rho ;\{m_i\},x_1,\cdots ,x_i)=m_i(\rho ;x_1,\cdots ,x_i)$, then $\{\tilde m_i\}_{i=2}^n$is a fibrewise $A_n$-form of $E_n[E]$.
Using the space $M_n[E]$, we can state the Lemma 5.7 in \cite{BV73} in the following form.
\begin{lem}[(Lemma 5.7 in \cite{BV73})]\label{class of An str}
Let $E$ be a fibrewise $A_n$-space and $\{m_i\},\{m_i'\}:B\rightarrow M_n[E]$ be fibrewise $A_n$-forms of $E$. 
Then the identity map of $E$ is a fibrewise $A_n$-equivalence between $\{m_i\}$ and $\{m_i'\}$ if and only if they are homotopic as sections of $M_n[E]\rightarrow B$.
\end{lem} 
\par 
Recall that there exists the universal fibration (not {\it principal}) $E_1\rightarrow M_1=BFG$ for Hurewicz fibrations with fibres homotopy equivalent to $G$, where $FG$ is the space of all self homotopy equivalences of $G$. 
For any space $B$ with homotopy type of a CW complex, there is the bijection between the free homotopy set $[B;M_1]$ and the set of all fibrewise homotopy equivalence classes of Hurewicz fibrations with fibres homotopy equivalent to $G$ given by pull-back of $E_1$.
Refer \cite{May75} for details.\par
We denote $M_n=M_n(G)=M_n[E_1]$ and $E_n=E_n(G)=E_n[E_1]$. 
\begin{prp}
Let $E$ be a fibrewise $A_n$-space over $B$. 
Then there exists a fibrewise $A_n$-map $f:E\rightarrow E_n$ over $\bar f:B\rightarrow M_n$ such that the induced map $E\rightarrow \bar f^*E_n$ is a fibrewise $A_n$-equivalence.
\end{prp}
\begin{proof}
Let $f_1:E\rightarrow E_1$ be a fibre map covering $\bar f_1:B\rightarrow M_1$ which induces a fibrewise homotopy equivalence $E\rightarrow \bar f_1^*E_1$. 
From Propositions \ref{equivalence1} and \ref{inverse}, there exists a fibrewise $A_n$-form of $\bar f_1^*E_1$ such that $E\rightarrow \bar f_1^*E_1$ is a fibrewise $A_n$-equivalence.
This fibrewise $A_n$-form gives a section of $\bar f_1^*M_n=M_n[\bar f_1^*E_1]\rightarrow B$.
This section and $f_1$ gives a fibrewise $A_n$-map $f:E\rightarrow E_n$ over $\bar f:B\rightarrow M_n$ which induces fibrewise $A_n$-equivalence $E\rightarrow \bar f^*E_n$.
\end{proof}
\begin{lem}\label{single ver}
Let $E$ be a fibrewise $A_n$-space over $B$ and $f,g:B\rightarrow M_n$ be maps such that $f^*E_n$ and $g^*E_n$ are fibrewise $A_n$-equivalent to $E$. Then $f$ and $g$ are homotopic.
\end{lem}
\begin{proof}
Let $p:M_n\rightarrow M_1$ be the projection. 
Then $pf$ and $pg$ are homotopic.
Let us denote this homotopy $h':I\times B\rightarrow M_0$. 
Since $p$ is a Hurewicz fibration, there exists a homotopy $h:I\times B\rightarrow M_n$ such that $ph=h'$ and $h_0=f$. 
Take a fibrewise $A_n$-map $\hat f:E\rightarrow E_n$ over $f$ which induces a fibrewise $A_n$-equivalence $E\rightarrow f^*E_n$.
From the homotopy lifting property of $E_n$, there exists a homotopy $\hat h:I\times E\rightarrow E_n$ such that $\hat h$ covers $h$ and $\hat h_0=\hat f$.
Here $E$ is fibrewise $A_n$-equivalent to both $g^*E_n$ and $h_1^*E_n$.
Since $ph_1=g$, they are the same fibrewise spaces.
From Propositions \ref{inverse}, \ref{composition} and \ref{homotopy fibrewise}, the identity map $h_1^*E_n\rightarrow g^*E_n$ is a fibrewise $A_n$-equivalence and Lemma \ref{class of An str} says that $h_1$ is homotopic to $g$. 
Therefore, $f$ and $g$ are homotopic.
\end{proof}
\par The next proposition follows from this lemma immediately.
\begin{prp}\label{class2}
Let $E$ and $E'$ be fibrewise $A_n$-equivalent fibrewise $A_n$-spaces over $B$ and $f,g:B\rightarrow M_n$ be maps.
If $E$ and $E'$ are fibrewise $A_n$-equivalent to $f^*E_n$ and $g^*E_n$ respectively, then $f$ and $g$ are homotopic.
\end{prp}
\par
From Proposition \ref{class1} and Proposition \ref{class2}, we conclude the classification theorem.
\begin{thm}[(the classification theorem for fibrewise $A_n$-spaces)]\label{classification}
Let $G$ be an $A_n$-space and $B$ be a space, both of which have homotopy types of CW complexes.
Fix a finite positive integer $n$.
Let us denote $A_n(G;B)$ the set of all fibrewise $A_n$-equivalent classes of fibrewise $A_n$-spaces over $B$ with fibres $A_n$-equivalent to $G$. 
Then the map $[B;M_n(G)]\rightarrow A_n(G;B)$ defined by pull-back of $E_n$ is well-defined and bijective.
\end{thm}
\par
We will construct the {\it associated principal fibration} of this universal fibration for our later calculation. 
Let $C_n$ be the fibrewise space over $M_n$ defined as the space consisting of all $A_n$-equivalences with its $A_n$-form from $G$ to some fibre of $E_n$:
\begin{align*}
C_n=\coprod_{b\in M_n}\left. \left\{\, \{f_i\} \in \prod_{i=1}^n \Map(\mathcal{J}_i\times G^{\times i},(E_n)_b)\, \right| 
\begin{array}{l}
\{f_i\}:{\rm an\,}A_n{\rm \mathchar`-form\, of\, an\,}A_n{\rm \mathchar`-equivalence}\\
{\rm from\,} G{\rm \, to\,}(E_n)_b 
\end{array}
\right\}.
\end{align*} 
As easily checked, the projection $C_n\rightarrow M_n$ is a Hurewicz fibration.
The fibres of $C_n\rightarrow M_n$ are homotopy equivalent to the space $F_{A_n}G$ consisting of all $A_n$-forms of self $A_n$-equivalences of $G$. 
It follows from the following proposition that $C_n$ is $\infty$-connected. 
\begin{prp}
Let $(G,\{m_i\}_{i=2}^n)$ be an $A_n$-space, $(F,\{\mu_i\}_{i=2}^{n-1})$ an $A_{n-1}$-space and $(f,\{f_i\}_{i=1}^{n-1}):G\to F$ an $A_{n-1}$-equivalence.
Then the space
\begin{align*}
\varPhi =\{ \, (\mu _n,f_n)\, |\, (f,\{f_i\}_{i=1}^n):(G,\{m_i\}_{i=2}^n)\to (F,\{\mu_i\}_{i=2}^n)\,{\rm is\, an}\,A_n{\rm \mathchar`-equivalence.}\, \}\\
\subset \Map (\mathcal{K}_n\times F^{\times n},F)\times \Map (\mathcal{J}_n\times G^{\times n},F)
\end{align*}
is contractible.
\end{prp}
\begin{proof}
For a fibrewise $Q^n W\mathcal{B}$-space $X,Y:Q^nW\mathcal{B}\to {\bf FWCG}$, we say $g:Q^nW(\mathcal{B}\otimes \mathcal{L}_1)\to {\bf FWCG}$ is a {\it fibrewise $Q^n\mathcal{B}$-map} from $X$ to $Y$ if $X=d^1f$ and $Y=d^0f$, where {\bf FWCG} denotes the category of fibrewise compactly generated spaces.
As remarked in \S 4, similar results in Chapter IV of \cite{BV73} hold for fibrewise $Q^n\mathcal{B}$-maps.
Let $\varPsi =\varPhi\times F$ be a fibrewise $A_n$-space over $\varPhi$ such that the $A_n$-form of the fibre $\varPsi_{(\mu_n,f_n)}=\{(\mu_n,f_n)\}\times F$ is $\{\mu_i\}_{i=2}^n$.
Then $1\times f:\varPhi\times (G,\{m_i\}_{i=2}^n)\to \varPsi$ is a fibrewise $A_n$-equivalence such that the $A_n$-form of $f_{(\mu_n,f_n)}:(G,\{m_i\}_{i=2}^n)\to \varPsi_{(\mu_n,f_n)}$ is $\{f_i\}_{i=1}^n$.
This fibrewise $A_n$-map $1\times f:\varPhi\times (G,\{m_i\}_{i=2}^n)\to \varPsi$ extends to a fibrewise $Q^n\mathcal{A}$-map $\hat f:Q^nW(\mathcal{A}\otimes \mathcal{L}_1)\to {\bf FWCG}$.
If we fix a point $(\mu_n^0,f_n^0)\in \varPhi$, the fibrewise map $1\times (f,\{f_1,\cdots ,f_{n-1},f_n^0\}):\varPhi \times (G,\{m_i\}_{i=2}^n)\to \varPhi \times (F,\{\mu_2,\cdots ,\mu_{n-1},\mu_n^0\})$ is a fibrewise $A_n$-equivalence and extends to a fibrewise $Q^n\mathcal{A}$-map $\hat {f^0}:Q^nW(\mathcal{A}\otimes \mathcal{L}_1)\to {\bf FWCG}$.
Then the indentity map $\varPsi \to \varPhi \times F$ gives a fibrewise $Q^n\mathcal{A}$-map $Q^nW(\mathcal{A}\otimes \mathcal{L}_1)\to {\bf FWCG}$.
\begin{align*}
\xymatrix{
\varPhi\times (G,\{m_i\}_{i=2}^n) \ar[d]^{\rm id} \ar[r]^-{\hat f} & \varPsi \ar[d] \\
\varPhi\times (G,\{m_i\}_{i=2}^n) \ar[r]^-{\hat{f^0}} & \varPhi \times (F,\{\mu_2,\cdots ,\mu_{n-1},\mu_n^0\})\\
}
\end{align*}
The above diagram commutes in the sense of the section 2 in Chapter IV of \cite{BV73}.
Then we can construct a fibrewise $Q^n(\mathcal{A}\otimes \mathcal{L}_1)$-map $h:Q^nW(\mathcal{A}\otimes \mathcal{L}_1\otimes \mathcal{L}_1)\to {\bf FWCG}$ from $\hat f$ to $\hat{f^0}$ whose underlying map is the identity map.
Thus, from Lemma 5.7 in \cite{BV73}, $\hat f,\hat {f^0}:Q^nW(\mathcal{A}\otimes \mathcal{L}_1)\to {\bf FWCG}$ are homotopic.
This implies that the space $\varPhi$ is contractible to a point $(\mu_n^0,f_n^0)$.
\end{proof}

\section{Fibrewise localization of fibrewise $A_n$-spaces}
We introduce the fibrewise $\mathcal{P}$-localization of fibrewise $A_n$-spaces.
In the following, notice our fibrewise spaces are Hurewicz fibrations.
\begin{dfn}
Let $E$ and $\bar E$ be fibrewise spaces over $B$ and $\mathcal{P}$ be a family of prime numbers. 
A fibrewise map $\ell:E\rightarrow \bar E$ is a {\it fibrewise $\mathcal{P}$-localization} if the restriction $\ell_b:E_b\rightarrow \bar E_b$ of $\ell$ to each fibre $E_b$ is a $\mathcal{P}$-localization. 
We say a fibrewise space is {\it fibrewise $\mathcal{P}$-local} if each of its fibres is $\mathcal{P}$-local.
\end{dfn}
\begin{rem}
Fibrewise localizations have the following universal property:
if $E'$ is a fibrewise $\mathcal{P}$-local space and $f:E\rightarrow E'$ is a fibrewise map, then there exists a fibrewise map $\bar f:\bar E\rightarrow E'$, unique up to fibrewise homotopy, such that $\bar f\ell$ is fibrewise homotopic to $f$, where the base space of $E$, $\bar E$ and $E'$ has homotopy type of a CW complex.
\end{rem}
\par
We quote Theorem 4.1 of \cite{May80} in the next form.
\begin{thm}
Let $G$ be a nilpotent finite CW complex and $\ell:G\rightarrow G_\mathcal{P}$ be the $\mathcal{P}$-localization of $G$. 
Then, $\ell$ induces a map $\ell_\# :BFG\rightarrow BFG_\mathcal{P}$ which corresponds to the map $FG\rightarrow FG_\mathcal{P}$ induced by $\ell$, which $\mathcal{P}$-localizes the identity component.
\end{thm}
\par
Let $G$ be a nilpotent finite CW complex and $E_1$ be a universal fibration of $G$. 
Then, we can construct the fibrewise $\mathcal{P}$-localization $E_1\rightarrow \bar E_1$ of $E_1$, where the classifying map of $\bar E_1$ is $\ell_\#$. 
Hence if $E$ is a fibrewise space with fibres homotopy equivalent to $G$ and with base $B$ homotopy equivalent to a CW complex, then we can take the fibrewise $\mathcal{P}$-localization $\ell:E\rightarrow \bar E$ of $E$. 
\begin{prp}\label{localization}
Let $E$ be a fibrewise $A_n$-space.
Then the fibrewise $\mathcal{P}$-localization $\ell :E\rightarrow \bar E$ has the fibrewise $A_n$-forms of $\bar E$ and $\ell$, both of which are unique up to homotopy through fibrewise $A_n$-forms.
For an $A_n$-space $G$ with homotopy type of a nilpotent finite complex, this fibrewise localization induces the map $M_n(G)\rightarrow M_n(G_\mathcal{P})$ and the following diagrams commute up to homotopy:
\begin{align*}
\xymatrix{
M_n(G) \ar[r] \ar[d] & \cdots \ar[r] & M_2(G) \ar[r] \ar[d] & BFG \ar[d]^{\ell_\#} \\
M_n(G_\mathcal{P}) \ar[r] & \cdots \ar[r] & M_2(G_\mathcal{P}) \ar[r] & BFG_\mathcal{P}\\
}
\end{align*}
\begin{align*}
\xymatrix{
\varOmega^{n-2}_0\Map(G^{\times n},G) \ar[r] \ar[d] & M_n(G) \ar[r] \ar[d] & M_{n-1}(G) \ar[d] \\
\varOmega^{n-2}_0\Map(G_\mathcal{P}^{\times n},G_\mathcal{P}) \ar[r] & M_n(G_\mathcal{P}) \ar[r] & M_{n-1}(G_\mathcal{P})\\
}
\end{align*}
where the map $\varOmega^{n-2}_0\Map(G^{\times n},G)\rightarrow \varOmega^{n-2}_0\Map(G_\mathcal{P}^{\times n},G_\mathcal{P})$ is the $\mathcal{P}$-localization.
\end{prp}
\begin{proof}
Let $\{m_i\}$ be the fibrewise $A_n$-form of $E$.
We construct fibrewise $A_n$-forms $\{\bar m_i\}$ and $\{\ell_i\}$ of $\bar E$ and $\ell$ respectively.
Let $\ell_1=\ell$.
Assume we have constructed fibrewise $A_{j-1}$-forms $\{\bar m_i\}_{i=2}^{j-1}$ and $\{\ell_i\}_{i=1}^{j-1}$.
The fibrewise map $\ell_j':(\mathcal{H}_j-{\rm Int}\, \delta(j,1,\cdots ,1))\times E^{\times_Bj}\rightarrow \bar E$ can be defined by using $\{m_i\}$, $\{\bar m_i\}$ and $\{\ell_i\}$ in a natural manner and be extended on $\mathcal{J}_j\times E^{\times_Bj}$.
We can also define $\bar m_j:\mathcal{L}_j\times \bar E^{\times_Bj}\rightarrow \bar E$ by using $\{\bar m_i\}$ and $\ell_j'(\delta(j,1,\cdots ,1)\times 1)|_{\mathcal{L}_j\times E^{\times_Bj}}=\bar m_j(1\times \ell^{\times_Bj})$.
There exists a fibrewise map $\bar m_j':\mathcal{K}_j\times \bar E^{\times_Bj}\rightarrow \bar E$ such that $\bar m_j'(1\times \ell^{\times_Bj})$ is fibrewise homotopic to $\ell_j'(\delta(j,1,\cdots ,1)\times 1)$.
Then $\bar m_j$ is fibrewise homotopic to $\bar m_j'|_{\mathcal{L}_j\times \bar E^{\times_Bj}}$.
Hence $\bar m_j$ can be extended on $\mathcal{K}_j\times \bar E^{\times_Bj}$ and there exists $\ell_j:\mathcal{J}_j\times E^{\times_Bj}\rightarrow \bar E$ such that $\{\ell_i\}_{i=1}^j$ is a fibrewise $A_j$-form of $\ell$.
\par
If there exist such $(\{\bar m_i^0\},\{\ell_i^0\})$ and $(\{\bar m_i^1\},\{\ell_i^1\})$, we can construct a homotopy $(\{\bar M_i\},\{L_i\})$ between them similarly.\par
The rest of this proposition follows immediately.
\end{proof}
\par
From the above proposition and the result of the previous section, we have the following homotopy commutative diagram:
\begin{align*}
\xymatrix{
\varOmega^{n-1}\Map(G^{\times n},G) \ar[r] \ar[d] & F_{A_n}G \ar[r] \ar[d] & F_{A_{n-1}}G \ar[d] \\
\varOmega^{n-1}\Map(G_\mathcal{P}^{\times n},G_\mathcal{P}) \ar[r] & F_{A_n}G_\mathcal{P} \ar[r] & F_{A_{n-1}}G_\mathcal{P}\\
}
\end{align*}
where the restriction $\varOmega^{n-1}_0\Map(G^{\times n},G)\to \varOmega^{n-1}_0\Map(G_\mathcal{P}^{\times n},G_\mathcal{P})$ is the $\mathcal{P}$-localization.
Since $FG\to FG_\mathcal{P}$ is the $\mathcal{P}$-localization of the identity component, one can show the following corollary inductively.
\begin{cor}\label{locstr}
The map $F_{A_n}G\rightarrow F_{A_n}G_\mathcal{P}$ induced from the map $M_n(G)\rightarrow M_n(G_\mathcal{P})$ above is the $\mathcal{P}$-localization of the identity component.
\end{cor}

\section{Fibrewise rationalization of automorphism bundles}
We review that the triviality of the fibrewise rationalization of an automorphism bundle by \cite{CS00}.
\par
First, we recall the construction of classifying spaces by using the geometric bar construction given by \cite{May72} and \cite{May75}.
Let $G$ be a topological group with non-degenerate basepoint and which has the homotopy type of a CW complex.
We can construct the universal bundle $EG=B(*,G,G)\rightarrow BG=B(*,G,*)$ by the geometric bar construction, where $EG$ is also a topological group and $G\subset EG$ is a closed subgroup compatible with its right $G$-action.
Then $BG$ is the coset space $EG/G$ and the adjoint action $G$ on $EG$ (i.e. $(g,x)\mapsto gxg^{-1}$) induces an action of $G$ on $BG$. 
We also call this action the adjoint action. 
We consider $G$ and $\varOmega BG$ as left $G$-space by this adjoint action.
\par
If $G$ is abelian, then $EG$ is also abelian.
Hence $BG$ is again a topological abelian group. 
Therefore, for any abelian group $\varGamma$, the Eilenberg-MacLane space of type $(\varGamma,n)$ can be taken as a topological abelian group.
\par
For a principal $G$-bundle $P$ over $B$, the {\it automorphism bundle} $\aut P$ of $P$ is the quotient space $P\times G/\sim$, where the equivalence relation $\sim $ is defined by $(u,x)\sim (ug,g^{-1}xg)$ for any $u\in P$ and $x,g\in G$.
The projection $P\rightarrow B$ induces a map $\aut P\rightarrow B$ and $\aut P$ is a fibre bundle with this projection.
Moreover, the multiplication $G\times G\rightarrow G$ induces a fibrewise map $\aut P\times_B\aut P\rightarrow \aut P$ and $\aut P$ becomes a fibrewise topological monoid, more precisely a {\it fibrewise topological group}.
\par
We show that $\aut EG$ and $EG\times_G\varOmega BG$ are fibrewise $A_\infty$-equivalent, where $EG\times_G\varOmega BG$ is the quotient space with identification $(u,\ell)\sim (ug,g^{-1}\ell)$ in $EG\times \varOmega BG$ and hence is also a fibrewise $A_\infty$-space.
\par
We use the notation in \cite{May72} for representing elements of $EG$ and $BG$.
The map $\tilde \zeta: EG\rightarrow PBG$ is defined by $\tilde \zeta (|[g_1,\cdots,g_j]g_{j+1},t|)(s)=|[g_1,\cdots,g_j,g_{j+1}],((1-s)t,s)|$.
Define $\zeta: G\rightarrow \varOmega BG$ by the restriction of $\tilde \zeta$.
This $\tilde \zeta$ is slightly different from May's $\tilde \zeta$ defined in \cite{May75} because May's $\zeta$ is not an $H$-map in general.
By definition, the following diagram commutes:
\begin{align*}
\xymatrix{
G \ar[r] \ar[d]_\zeta & EG \ar[r] \ar[d]^{\tilde \zeta} & BG \ar@{=}[d] \\
\varOmega BG \ar[r] & PBG \ar[r] & BG \\
}
\end{align*}
The map $\zeta$ is $G$-equivariant and is a homotopy equivalence.
\begin{lem}
The map $\zeta$ has a $G$-equivariant $A_\infty$-form, where an $A_\infty$-form $\{\zeta_i\}_{i=1}^\infty$ is {\it $G$-equivariant} if $\zeta_i(1\times g^{\times i})=g\zeta_i$ for each $g\in G$ and $i$.
\end{lem}
\begin{proof}
For $g_1,\cdots ,g_n\in G$ and $\rho \in \mathcal{K}_i$, we have
\begin{align*}
\zeta (g_1,\cdots ,g_n)(s)=|[g_1\cdots g_n](1-s,s)|=|[g_1,\cdots ,g_n](1-s,0,\cdots ,0,s)|.
\end{align*}
Let $\{\omega^i:\mathcal{K}_i\to P_i\}_{i=2}^\infty$ and $\{m_i:\mathcal{K}_i\times (\varOmega BG)^{\times i}\to \varOmega BG\}_{i=2}^\infty$ as in Example (iv) in \S 4.
Then, for $\rho \in \mathcal{K}_n$, $g_1,\cdots ,g_n \in G$ and $\omega^n_{k-1}(\rho)\leq s\leq \omega^n_k(\rho)$, we obtain
\begin{align*}
m_i(\rho ;\zeta (g_1),\cdots ,\zeta (g_n))(s)=&\left|[g_k],\left(1-\frac{s-\omega^n_{k-1}(\rho)}{\omega^n_k(\rho)-\omega^n_{k-1}(\rho)}, \frac{s-\omega^n_{k-1}(\rho)}{\omega^n_k(\rho)-\omega^n_{k-1}(\rho)}\right)\right|\\
=&\left|[g_1,\cdots ,g_n],\left(0,\cdots ,0,1-\frac{s-\omega^n_{k-1}(\rho)}{\omega^n_k(\rho)-\omega^n_{k-1}(\rho)}, \frac{s-\omega^n_{k-1}(\rho)}{\omega^n_k(\rho)-\omega^n_{k-1}(\rho)},0,\cdots ,0\right)\right|.
\end{align*}
From these equations, it is sufficient for us to construct a fibrewise $A_\infty$-form $\{\zeta_i\}_{i=1}^\infty$ such that $\zeta_n(\rho;g_1,\cdots ,g_n)(s)=|[g_1,\cdots ,g_n],\xi_n(\rho;s)|$ for some $\xi_n:\mathcal{J}_n\times I\to \varDelta^n$.
Since each simplex $\varDelta^n$ is contractible, such $\{\xi_i\}_{i=1}^\infty$ can be constructed inductively.
\end{proof}
\par
Hence a fibrewise map $1\times_G \zeta: EG\times_G G\rightarrow EG\times_G\varOmega BG$ is a fibrewise $A_\infty$-equivalence.
\par
Here, $EG\times_G\varOmega BG$ is the fibrewise based loop space of the fibrewise pointed space $EG\times_GBG$ with section $BG\rightarrow EG\times_GBG$ given by $[u]\mapsto [u,e]$ for $u\in EG$ and the identity element $e\in EG$.
And the map $EG\times EG\rightarrow EG\times EG$ $(u,u')\mapsto (u,uu')$ induces a fibrewise pointed topological equivalence $EG\times_GBG\rightarrow BG\times BG$, whose target is the fibrewise pointed space $BG\stackrel{\sigma}{\rightarrow}BG\times BG\stackrel{\pi}{\rightarrow}BG$ such that $\pi$ is the first projection and $\sigma$ is the diagonal map. 
\par
Let $G$ be a compact connected Lie group.
Since the rational cohomology ring $H^*(BG;\bm{Q})$ of $BG$ is a polynomial ring, there exists the rationalization $\ell:BG\rightarrow (BG)_{(0)}$, where $(BG)_{(0)}$ is a topological abelian group.
The map $1\times \ell:BG\times BG\rightarrow BG\times (BG)_{(0)}$ is a fibrewise rationalization over $BG$, where $BG\times (BG)_{(0)}$ is a fibrewise pointed space with projection given by the first projection and with section $(1,\ell):BG\rightarrow BG\times (BG)_{(0)}$. 
Finally, we have the fibrewise pointed topological equivalence $BG\times (BG)_{(0)}\rightarrow BG\times (BG)_{(0)}$ $(x,y)\mapsto (x,\ell(x)^{-1}y)$, whose target is a fibrewise pointed space with section given by $x\mapsto (x,e)$. 
Thus $EG\times_G\varOmega BG\rightarrow EG\times \varOmega (BG)_{(0)}$ is a fibrewise rationalization and is a fibrewise $A_\infty$-homomorphism.
Moreover, $EG\times \varOmega (BG)_{(0)}$ is trivial as a fibrewise $A_\infty$-space.
\begin{thm}\label{triviality}
Let $G$ be a compact connected Lie group, $B$ be a space with homotopy type of a CW complex and $P$ be a principal $G$-bundle over $B$.
If $f:B\rightarrow M_n(G)$ is the classifying map of the automorphism bundle $\aut P=P\times_GG$ of $P$, then the composition of $f$ and $M_n(G)\rightarrow M_n(G_{(0)})$ is null-homotopic.
\end{thm}

\section{Proof of Theorem \ref{mainthm}}
We state Lemma 6.4 of \cite{CS00} as the following lemma.
\begin{lem}\label{diagram}
Let $A_1,A_2,B_1,B_2$ be groups, $A_2,B_2$ be abelian, $A_2$ be finitely generated, $A_3$ and $B_3$ be sets on which $A_2$ and $B_2$ act respectively, $A_4$ and $B_4$ be sets and $a_3\in A_3$ be a fixed element.
Consider the following commutative diagram.
\begin{align*}
\xymatrix{
A_1 \ar[r]^-{f_1} \ar[d]^-{\ell_1} & A_2 \ar[r]^-{f_2} \ar[d]^-{\ell_2} & A_3 \ar[r]^-{f_3} \ar[d]^-{\ell_3} & A_4 \ar[d]^-{\ell_4} \\
B_1 \ar[r]^-{g_1} & B_2 \ar[r]^-{g_2} & B_3 \ar[r]^-{g_3} & B_4 \\ 
}
\end{align*}
Assume the following conditions:\\
(i) The maps $f_1,g_1,\ell_1,\ell_2$ are homomorphisms, where $\ell_1$ is $\bm{Q}$-surjective (i.e. for any $b\in B_1$, there exists an integer $n$ such that $b^n\in \image \ell_1$) and $\ker \ell_2$ is finite.\\
(ii) For any $a_2\in A_2$ and $a\in A_3$, $\ell_3(a_2\cdot a)=\ell_2(a_2)\cdot \ell_3(a)$.\\
(iii) For any $a\in A_3$, $A_2\cdot a=f_3^{-1}(f_3(a))$. Similarly, for any $b \in B_3$, $B_2\cdot b=g_3^{-1}(g_3(b))$.\\
(iv) The isotropy subgroup of $a_3$ is $\image f_1$, and the isotropy subgroup of $\ell_3(a_3)$ is $\image g_1$.\\
Then $f_3^{-1}(f_3(a_3))\cap \ell_3^{-1}(\ell_3(a_3))$ is finite. Moreover, if $\ell_4$ is finite-to-one and if the condition (iv) holds for any $a_3\in A_3$, then $\ell_3$ is also finite-to-one.
\end{lem}
\begin{proof}
Fix $a\in A_2$ such that $a\cdot a_3\in \ell_3^{-1}(\ell_3(a_3))$. 
From (ii) and (iii), there exists $b\in B_1$ such that $g_1(b)=\ell_2(a)$. 
From (i), we can take an integer $n$ and $a'\in A_1$ such that $\ell_1(a')=b^n$. 
Since $\ell_2(na-f_1(a'))=0$, the order of $na-f_1(a')$ is finite. 
Hence, for some integer $k$, $(ka)\cdot a_3=a_3$. 
Then, because $A_2$ is finitely generated abelian group, the condition (iii) says that $f_3^{-1}(f_3(a_3))\cap \ell_3^{-1}(\ell_3(a_3))$ is finite. 
The rest of this lemma follows immediately from this assertion.
\end{proof}
\par
We also quote Theorem 6.2 and Corollary 5.4 in Chapter II of \cite{HMR75}.
Let $[X;Y]^*$ denote the homotopy set of basepoint-preserving maps from $X$ to $Y$. 
\begin{thm}\label{HMR1}
Let $X$ be an $H$-space with non-degenerate basepoint and which has the homotopy type of a connected CW complex and $\ell:X\rightarrow X_\mathcal{P}$ be the $\mathcal{P}$-localization $H$-map. 
Then, for a connected finite complex $W$, $\ell_*:[W;X]^*\rightarrow [W;X_\mathcal{P}]^*$ is $\bm{Q}$-bijective, in other words, $\ell_*$ is $\bm{Q}$-surjective and every element in $\ker \ell_*$ has a finite order.
\end{thm}
\par
\begin{lem}\label{HMR2}
Let $G$ be a nilpotent finite complex and $\ell:G\rightarrow G_\mathcal{P}$ be the $\mathcal{P}$-localization.
Then the homomorphism $\pi_0(FG)\rightarrow \pi_0(FG_\mathcal{P})$ induced from $\ell$ is finite-to-one.
\end{lem}
\par
In the following argument, assume $B$ is a connected finite complex, $(G,\{m_i\}_{i=2}^n)$ is an $A_n$-space with homotopy type of a connected finite complex such that $m_2:G\times G\rightarrow G$ has the homotopy unit and $G_\mathcal{P}$ is the $\mathcal{P}$-localization of $G$.\par
From Corollary \ref{locstr} and Theorem \ref{HMR1}, $[\varSigma B;M_n(G)]^*\rightarrow [\varSigma B;M_n(G_\mathcal{P})]^*$ is $\bm{Q}$-surjective.
Since $G$ is an $H$-space and is a finite complex, $\pi_r(M_n(G))$ is finitely generated.
Then $\pi_r(M_n(G))\rightarrow \pi_r(M_n(G_\mathcal{P}))$ is finite-to-one for $r\geq 2$.
In fact, this also holds for $r=1$.
\begin{lem}\label{pi1}
$\pi_1(M_n(G))\rightarrow \pi_1(M_n(G_\mathcal{P}))$ is finite-to-one.
\end{lem}
\begin{proof}
Since we have the following commutative diagram, we show $\pi_0(F_{A_n}G)\rightarrow \pi_0(F_{A_n}G_\mathcal{P})$ is finite-to-one.
\begin{align*}
\xymatrix{
\pi _1(M_n(G)) \ar[d] \ar[r]^\sim & \pi _0(F_nG) \ar[d] \\
\pi _1(M_n(G_\mathcal{P})) \ar[r]^\sim & \pi _0(F_nG_\mathcal{P})\\
}
\end{align*}
When $n=1$, this follows by Lemma \ref{HMR2}. 
Assume this holds for $n-1$, $n\geq 2$.
The diagram in Proposition \ref{localization} yields the following commutative ladder:
\begin{align*}
\xymatrix{
\pi_1(F_{A_{n-1}}G) \ar[r] \ar[d] & \pi_0(\varOmega^{n-1}\Map(G^{\times n},G)) \ar[r] \ar[d] & \pi_0(F_{A_n}G) \ar[r] \ar[d] & \pi_0(F_{A_{n-1}}G) \ar[d] \\
\pi_1(F_{A_{n-1}}G_\mathcal{P}) \ar[r] & \pi_0(\varOmega^{n-1}\Map(G_\mathcal{P}^{\times n},G_\mathcal{P})) \ar[r] & \pi_0(F_{A_n}G_\mathcal{P}) \ar[r] & \pi_0(F_{A_{n-1}}G_\mathcal{P})\\
}
\end{align*}
Let us verify the conditions of Lemma \ref{diagram} about this diagram.
The abelian group $\pi_0(\varOmega^{n-1}\Map(G^{\times n},G))$ acts on  $\pi_0(F_{A_n}G)$ through the above homomorphism.
Similarly, $\pi_0(\varOmega^{n-1}\Map(G_\mathcal{P}^{\times n},G_\mathcal{P}))$ acts on $\pi_0(F_{A_n}G_\mathcal{P})$.
Since $G$ is a finite complex and is an $H$-space, $\pi_0(\varOmega^{n-1}\Map(G^{\times n},G))$ is a finitely generated abelian group.
From this and Theorem \ref{HMR1}, the kernel of $\pi_0(\varOmega^{n-1}\Map(G^{\times n},G))\to \pi_0(\varOmega^{n-1}\Map(G_\mathcal{P}^{\times n},G_\mathcal{P}))$ is finite.
The map $\pi_2(M_{n-1}(G))\to \pi_2(M_{n-1}(G_\mathcal{P}))$ is $\bm{Q}$-surjective from the above, then $\pi_1(F_{A_{n-1}}G)\to \pi_1(F_{A_{n-1}}G_\mathcal{P})$ is $\bm{Q}$-surjective.
Thus the condition (i) is verified.
The conditions (ii), (iii) and (iv) (for any $a_3\in \pi_0(F_{A_n}G)$) are easily verified.
Hence $\pi_0(F_{A_n}G)\rightarrow \pi_0(F_{A_n}G_\mathcal{P})$ is finite-to-one.
\end{proof}
\par Now we show the following proposition.
\begin{prp}\label{finone}
$[B;M_n(G)]^*\rightarrow [B;M_n(G_\mathcal{P})]^*$ is finite-to-one.
\end{prp}
\begin{proof}
When $B$ is a point, this is trivial. 
When $B$ is 1-dimensional, $B$ is a wedge sum of finite circles.
Then this proposition follows by Lemma \ref{pi1}. 
Assume $B$ is a complex given by attaching the complex $B'$ to one $r$-cell ($r\geq 2$), where the assertion above holds for the complex $B'$. 
Let us consider the following commutative diagram given by the cofibration $B'\rightarrow B\rightarrow S^r$:
\begin{align*}
\xymatrix{
[\varSigma B';M_n(G)]^* \ar[r] \ar[d] & \pi_r(M_n(G)) \ar[r] \ar[d] & [B;M_n(G)]^* \ar[r] \ar[d] & [B';M_n(G)]^* \ar[d] \\ 
[\varSigma B';M_n(G_\mathcal{P})]^* \ar[r] & \pi_r(M_n(G_\mathcal{P})) \ar[r] & [B;M_n(G_\mathcal{P})]^* \ar[r] & [B';M_n(G_\mathcal{P})]^* \\ 
}
\end{align*}
This cofibration induces the actions of $\pi_r(M_n(G))$ on $[B;M_n(G)]^*$ and $\pi_r(M_n(G_\mathcal{P}))$ on $[B;M_n(G_\mathcal{P})]^*$.
Since we can apply Lemma \ref{diagram} to this diagram, $[B,M_n(G)]^*\rightarrow [B,M_n(G_\mathcal{P})]^*$ is finite-to-one.
\end{proof}
\par Here, we consider the case that $G$ is a compact connected Lie group.
\begin{thm}\label{thm}
Let $B$ be a finite connected complex and $G$ be a compact connected Lie group. 
For each $n<\infty$, the number of fibrewise $A_n$-equivalent classes represented by the automorphism bundles of principal $G$-bundle over $B$ is finite.
\end{thm}
\begin{proof}
From Proposition \ref{triviality}, the classifying map of the fibrewise rationalization of an automorphism bundle is null-homotopic. 
Then this classifying map is also null-homotopic preserving basepoint. 
From Proposition \ref{finone}, there exist only finitely many classifying maps $B\rightarrow M_n(G)$ up to homotopy which correspond to some automorphism bundle.
Hence the conclusion follows.
\end{proof}
\par In general, if $(E,\{m_i\})$ is a fibrewise $A_n$-space, then the space of all sections $\varGamma(E)$ of $E$ is naturally an $A_n$-space, where the $A_n$-form $\{\varGamma m_i\}$ of $\varGamma(E)$ is given by $\varGamma m_i(\rho ;\varphi_1,\cdots,\varphi_i)(b)=m_i(\rho;\varphi_1(b),\cdots,\varphi_i(b))$.
If $E$ and $E'$ is a fibrewise $A_n$-space over $B$ and $f:E\rightarrow E'$ is a fibrewise $A_n$-map, then the map $\varGamma f:\varGamma (E)\rightarrow \varGamma (E')$ given by $(\varGamma f)(\varphi)(b)=f\varphi(b)$ is an $A_n$-map.
Moreover, if $f$ is a fibrewise $A_n$-equivalence, then $\varGamma f$ is an $A_n$-equivalence.
\par
Let $G$ be a compact connected Lie group.
Since $\mathcal{G}(P)$ is isomorphic to $\varGamma(\aut P)$ for any principal bundle $P$, Theorem \ref{thm} implies Theorem \ref{mainthm}.

\section{A counterexample when $n=\infty$}
In Theorem \ref{mainthm}, it is essential to assume $n$ is finite.
Kono and Tsukuda have shown the following theorem:
\begin{thm}[(\cite{KT00} and \cite{Tsu01})]
Let $X$ be an oriented simply connected closed 4-manifold and $P$ and $P'$ be principal $\SU(2)$-bundles over $X$.
The classifying spaces $B\mathcal{G}(P)$ and $B\mathcal{G}(P')$ are homotopy equivalent if and only if
\begin{align*}
\left\{
\begin{array}{ll}
|c_2(P)[X]|=|c_2(P')[X]| & {\rm if}\, X\, {\rm admits\, an\, orientation\, reversing\, homotopy\, equivalence,}\\
c_2(P)[X]=c_2(P')[X] & {\rm otherwise},\\
\end{array}
\right.
\end{align*}
where $c_2(P)[X]$ represents the coupling of the second Chern class of $P$ and the fundamental class of $X$. 
\end{thm}
\par
Stasheff has shown in \cite{Sta63b} that a map between topological monoids is an $A_n$-map if and only if it admits an appropriate map between $A_n$-structures.
From this, $\mathcal{G}(P)$ and $\mathcal{G}(P')$ are $A_\infty$-equivalent if and only if $B\mathcal{G}(P)$ and $B\mathcal{G}(P')$ are homotopy equivalent.
Hence we conclude the following counterexample when $n=\infty$.
\begin{prp}\label{counterexample}
For a simply connected closed 4-manifold $X$, there are infinite distinct $A_\infty$-types of gauge groups of principal $\SU(2)$-bundles over $X$.
\end{prp}

\section{The gauge groups of principal $\SU(2)$-bundles over $S^4$}
In general, it is difficult problem to count the number of $A_n$-types of gauge groups.
However, one can often partially know the behavior of the localizations of $A_n$-types of gauge groups.
\par 
Denote the principal $\SU(2)$-bundle over $S^4$ with second Chern class $k\in H^4(S^4)\simeq \bm{Z}$ by $P_k$.
For a family of prime numbers $\mathcal{P}$, let $P_{k,\mathcal{P}}$ be the principal $\SU(2)_\mathcal{P}$-bundle over $S^4$ with second Chern class $k\in H^4(S^4;\bm{Z}_\mathcal{P})\simeq \bm{Z}_\mathcal{P}$, where $\bm{Z}_\mathcal{P}$ is the localization of the ring $\bm{Z}$ at $\mathcal{P}$. 
Here we remark that $\SU(2)_\mathcal{P}$ can be taken as a topological group with homotopy type of a CW complex \cite{Mil56}.
Then, using the same construction as in \S 7, the fibrewise $\mathcal{P}$-localization of $\aut P_k$ as a fibrewise $A_\infty$-space is $\aut P_{k,\mathcal{P}}$.
\par
Denote the identity component of the gauge group $\mathcal{G}(P_{k,\mathcal{P}})$ by $\mathcal{G}_0(P_{k,\mathcal{P}})$ and the kernel of the evaluation at the basepoint $ev:\mathcal{G}_0(P_{k,\mathcal{P}})\rightarrow \SU(2)_\mathcal{P}$ by $\mathcal{G}_{0,0}(P_{k,\mathcal{P}})$.
Atiyah and Bott \cite{AB83} constructed the universal bundle of the gauge group
\begin{align*}
\mathcal{G}(P_{k,\mathcal{P}})\rightarrow E\mathcal{G}(P_{k,\mathcal{P}})\rightarrow \Map(S^4;\bm{H}P^\infty_\mathcal{P};k),
\end{align*}
where $E\mathcal{G}(P_{k,\mathcal{P}})$ is the space of all bundle maps from $P_{k,\mathcal{P}}$ to the universal bundle over $\bm{H}P^\infty_\mathcal{P}$ and $\Map(S^4;\bm{H}P^\infty_\mathcal{P};k)$ is the path component of $\Map(S^4;\bm{H}P^\infty_\mathcal{P})$ corresponding to $\ell k:S^4\rightarrow \bm{H}P^\infty_\mathcal{P}$.
Then, the $\mathcal{P}$-localization of $\mathcal{G}_0(P_k)$ as an $A_\infty$-space is $\mathcal{G}_0(P_{k,\mathcal{P}})$.
\begin{dfn}
A fibrewise $A_n$-space $E$ over $B$ is said to be {\it trivial} if $E$ is fibrewise $A_n$-equivalent to the fibrewise $A_n$-space $B\times E_b$ for some $b\in B$.
\end{dfn}
\par
As in \cite{KK}, $\aut P_{k,\mathcal{P}}$ is trivial as a fibrewise $A_n$-space if and only if there exists a map $f:S^4\times \bm{H}P^{n}_\mathcal{P}\rightarrow \bm{H}P^\infty_\mathcal{P}$ such that the following diagram commutes up to homotopy:
\begin{align*}
\xymatrix{
S^4 \vee \bm{H}P^n_\mathcal{P} \ar[r]^-{\ell k\vee i} \ar[d]_{j} & \bm{H}P^\infty_\mathcal{P}\vee \bm{H}P^\infty_\mathcal{P} \ar[d]^{\nabla} \\
S^4 \times \bm{H}P^n_\mathcal{P} \ar[r]^-{f} & \bm{H}P^\infty_\mathcal{P}
}
\end{align*}
where $k:S^4\rightarrow \bm{H}P^\infty$ is the classifying map of $P_k$, $\ell:\bm{H}P^\infty \rightarrow \bm{H}P^\infty_\mathcal{P}$ is the $\mathcal{P}$-localization of $\bm{H}P^\infty$, $i$ and $j$ are inclusions and $\nabla$ is the folding map.
Moreover, the following proposition holds.
\begin{prp}\label{trivialgaugegroup}
The gauge group $\mathcal{G}(P_{k,\mathcal{P}})$ of $P_{k,\mathcal{P}}$ is $A_n$-equivalent to the gauge group $\mathcal{G}(S^4\times \SU(2)_\mathcal{P})$ of the trivial bundle if and only if $\aut P_{k,\mathcal{P}}$ is trivial as a fibrewise $A_n$-space. 
\end{prp}
\begin{proof}
We identify the gauge group $\mathcal{G}(P_{k,\mathcal{P}})$ with the space $\varGamma (\aut P_{k,\mathcal{P}})$ of sections.
Let $F:\mathcal{G}(S^4\times \SU(2)_\mathcal{P})\rightarrow \mathcal{G}(P_{k,\mathcal{P}})$ be an $A_n$-equivalence and define the fibrewise $A_n$-map $f:B\times \SU(2)_\mathcal{P}\rightarrow \aut P_{k,\mathcal{P}}$ by $f(b,g)=f(s(g))(b)$, where $s:\SU(2)_\mathcal{P}\rightarrow \mathcal{G}(S^4\times \SU(2)_\mathcal{P})$ is the standard section given by $s(g)(b)=(b,g)$.
Let us see that $f$ is a fibrewise $A_n$-equivalence, equivalently, $ev F s:\SU(2)_\mathcal{P}\rightarrow \SU(2)_\mathcal{P}$ is a homotopy equivalence.
Consider the following evaluation fibration:
\begin{align*}
\mathcal{G}_{0,0}(P_{k,\mathcal{P}})\rightarrow \mathcal{G}_0(P_{k,\mathcal{P}})\rightarrow \SU(2)_\mathcal{P}.
\end{align*}
Note that $\pi_i(\mathcal{G}_{0,0}(S^4\times \SU(2)_\mathcal{P}))$ and $\pi_i(\mathcal{G}_{0,0}(P_{k,\mathcal{P}}))$ are isomorphic finite groups for each $i$.
Since $F_*:\pi_2(\mathcal{G}_0(S^4\times \SU(2)_\mathcal{P}))\rightarrow \pi_2(\mathcal{G}_0(P_{k,\mathcal{P}}))$ is an isomorphism and these groups are finite, $\pi_2(\mathcal{G}_{0,0}(P_{k,\mathcal{P}}))\rightarrow \pi_2(\mathcal{G}_0(P_{k,\mathcal{P}}))$ is also an isomorphism.
Hence $ev_*:\pi_3(\mathcal{G}_0(P_{k,\mathcal{P}}))\rightarrow \pi_3(\SU(2)_\mathcal{P})$ has a section.
Then one can see that $(evFs)_*:\pi_3(\SU(2)_\mathcal{P})\rightarrow \pi_3(\SU(2)_\mathcal{P})$ is an isomorphism.
Therefore, $ev F s:\SU(2)_\mathcal{P}\rightarrow \SU(2)_\mathcal{P}$ is a homotopy equivalence.
\end{proof}
\begin{prp}[(Lemma 2.2 in \cite{Tsu01})]\label{primetop}
If $r\in \bm{Z}$ is prime to every element of $\mathcal{P}$, then $\mathcal{G}(P_{kr,\mathcal{P}})$ is $A_\infty$-equivalent to $\mathcal{G}(P_{k,\mathcal{P}})$.
\end{prp}
\begin{proof}
The map $S^4\rightarrow S^4$ with degree $r$ induces the following homotopy commutative diagram:
\begin{align*}
\xymatrix{
\varOmega_0^4\bm{H}P^\infty _\mathcal{P} \ar[d] \ar[r] & \Map (S^4;\bm{H}P^\infty _\mathcal{P};kr) \ar[d] \ar[r] & \bm{H}P^\infty _\mathcal{P} \ar@{=}[d] \\
\varOmega_0^4\bm{H}P^\infty _\mathcal{P} \ar[r] & \Map (S^4;\bm{H}P^\infty _\mathcal{P};k) \ar[r] & \bm{H}P^\infty _\mathcal{P} \\
}
\end{align*}
where $\varOmega_0^4\bm{H}P^\infty _\mathcal{P}\rightarrow \varOmega_0^4\bm{H}P^\infty _\mathcal{P}$ is a homotopy equivalence.
Hence $\Map (S^4;\bm{H}P^\infty _\mathcal{P};kr)$ and $\Map (S^4;\bm{H}P^\infty _\mathcal{P};k)$ are homotopy equivalent.
Therefore, $\mathcal{G}(P_{kr,\mathcal{P}})$ and $\mathcal{G}(P_{k,\mathcal{P}})$ are $A_\infty$-equivalent.
\end{proof}
\par
In the rest of this section, we give an observation on the fibrewise $A_n$-types, using a result of \cite{Tod65} concerning with the homotopy groups of $S^3$.
\begin{thm}\label{p-1/2}
Let $p=\min \mathcal{P}$.
If $p\geq 3$, then $\aut P_{k,\mathcal{P}}$ is trivial as a fibrewise $A_{\frac{p-1}{2}-1}$-space.
Moreover, $\aut P_{k,\mathcal{P}}$ is trivial as a fibrewise $A_{\frac{p-1}{2}}$-space if and only if $k \equiv 0 \mod p$.
\end{thm}
\begin{proof}
Let $x\in H^4(S^4;\bm{Z}_\mathcal{P})$ and $c\in H^4(\bm{H}P^n_\mathcal{P};\bm{Z}_\mathcal{P})$ be generators such that $c|_{S^4}=x$ and $u\in H^4(K(\bm{Z}_\mathcal{P},4);\bm{Z}_\mathcal{P})$ be the fundamental class.
From obstruction theory, we can take maps $f:S^4\times \bm{H}P^n_\mathcal{P}\to K(\bm{Z}_\mathcal{P},4)$ and $g:\bm{H}P^\infty_\mathcal{P}\to K(\bm{Z}_\mathcal{P},4)$ such that $f^*u=kx\times 1+1\times c\in H^4(S^4 \times \bm{H}P^n_\mathcal{P};\bm{Z}_\mathcal{P})$ and $g^*u=c\in H^4(\bm{H}P^\infty_\mathcal{P};\bm{Z}_\mathcal{P})$.
Since $j^*f^*u=(\ell k\vee i)^*\nabla ^*g^*=kx\vee 1+1\vee c\in H^4(S^4\vee \bm{H}P^n_\mathcal{P};\bm{Z}_\mathcal{P})$, the following digram commutes:
\begin{align*}
\xymatrix{
S^4 \vee \bm{H}P^n_\mathcal{P} \ar[r]^-{\ell k\vee i} \ar[d]_{j} & \bm{H}P^\infty_\mathcal{P}\vee \bm{H}P^\infty_\mathcal{P} \ar[d]^{\nabla} \\
S^4 \times \bm{H}P^n_\mathcal{P} \ar[rd]_-{f} & \bm{H}P^\infty_\mathcal{P} \ar[d]^-{g} \\
 & K(\bm{Z}_\mathcal{P},4) 
}
\end{align*}
Since $g^*:H^i(K(\bm{Z}_\mathcal{P},4);\bm{Z}_\mathcal{P})\rightarrow H^i(\bm{H}P^\infty_\mathcal{P};\bm{Z}_\mathcal{P})$ is an isomorphism for $i<2p+2$, if $n\leq (p-1)/2-1$, there exists a map $f':S^4 \times \bm{H}P^n_\mathcal{P}\rightarrow \bm{H}P^\infty_\mathcal{P}$ such that $f'j$ and $\nabla(\ell k\vee i)$ are homotopic.
This implies the first half of this theorem.
\par
Assume that there exists a following homotopy commutative diagram:
\begin{align*}
\xymatrix{
S^4 \vee \bm{H}P^{(p-1)/2}_\mathcal{P} \ar[r]^-{\ell k\vee i} \ar[d]_{j} & \bm{H}P^\infty_\mathcal{P}\vee \bm{H}P^\infty_\mathcal{P} \ar[d]^{\nabla} \\
S^4 \times \bm{H}P^{(p-1)/2}_\mathcal{P} \ar[r]^-{f} & \bm{H}P^\infty_\mathcal{P}
}
\end{align*}
Then $f^*\mathcal{P}^1u\in H^{2p+2}(S^4 \times \bm{H}P^{(p-1)/2}_\mathcal{P};\bm{Z}/p\bm{Z})$ is computed as follows:
\begin{align*}
f^*\mathcal{P}^1c=\mathcal{P}^1(kx\times 1+1\times c)=k\mathcal{P}^1x\times 1+1\times \mathcal{P}^1c=0.
\end{align*}
On the other hand, since $\mathcal{P}^1c=\pm 2c^{(p+1)/2}\in H^{2p+2}(\bm{H}P^\infty_\mathcal{P};\bm{Z}/p\bm{Z})$,
\begin{align*}
f^*\mathcal{P}^1c=\pm 2f^*c^{(p+1)/2}=\pm 2(kx\times 1+1\times c)^{(p+1)/2}=\pm k(p+1)x\times c^{(p-1)/2}.
\end{align*}
Therefore, $k\equiv 0 \mod p$.
\par
Conversely, suppose $k\equiv 0 \mod p$.
Then there is some integer $r\in \bm{Z}$ such that $k=pr$.
From the first half, we can take a map $f':S^4 \times \bm{H}P^{(p-3)/2}_\mathcal{P}\cup *\times \bm{H}P^{(p-1)/2}_\mathcal{P}\rightarrow \bm{H}P^\infty_\mathcal{P}$ such that the following diagram commutes up to homotopy:
\begin{align*}
\xymatrix{
S^4 \vee \bm{H}P^{(p-1)/2}_\mathcal{P} \ar[r]^-{p\vee 1} \ar[d]_{j} & S^4 \vee \bm{H}P^{(p-1)/2}_\mathcal{P} \ar[r]^-{\ell r\vee i} \ar[d]_{j} & \bm{H}P^\infty_\mathcal{P}\vee \bm{H}P^\infty_\mathcal{P} \ar[d]^{\nabla} \\
S^4 \times \bm{H}P^{(p-3)/2}_\mathcal{P}\cup *\times \bm{H}P^{(p-1)/2}_\mathcal{P} \ar[r]^-{p\times 1\cup *\times 1} & S^4 \times \bm{H}P^{(p-3)/2}_\mathcal{P}\cup *\times \bm{H}P^{(p-1)/2}_\mathcal{P} \ar[r]^-{f'} & \bm{H}P^\infty_\mathcal{P}
}
\end{align*}
Now, the obstruction to extending $f'$ over $S^4\times \bm{H}P^{(p-1)/2}_\mathcal{P}$ lives in $\pi_{2p+1}(\bm{H}P^\infty_\mathcal{P})$.
Since $\pi_{2p+1}(\bm{H}P^\infty_\mathcal{P})\simeq \bm{Z}/p\bm{Z}$, there is no obstruction to extending  $f'(p\times 1\cup *\times 1)$ over $S^4\times \bm{H}P^{(p-1)/2}_\mathcal{P}$.
Hence $\aut P_{k,\mathcal{P}}$ is trivial as a fibrewise $A_{\frac{p-1}{2}}$-space.
\end{proof}
\par
Furthermore, for $r\in \bm{Z}$ prime to each element of $\mathcal{P}$, we see that $\aut P_{pr,\mathcal{P}}$ is not trivial as a fibrewise $A_{p-1}$-space.
\par
Let $\zeta$ be the universal $\SU(2)$-bundle over $\bm{H}P^\infty$. Then,
\begin{align*}
K(\bm{H}P^\infty_\mathcal{P})_\mathcal{P}=\bm{Z}_\mathcal{P}[a],
\end{align*}
where $K(\cdot)_\mathcal{P}$ represents the $\mathcal{P}$-local complex $K$-theory and $a=\zeta-2$.
Let $b=-c_2(\zeta)\in H^4(\bm{H}P^\infty_\mathcal{P};\bm{Q})$ then
\begin{align*}
ch\,a=\sum_{j=1}^\infty\frac{2b^j}{(2j)!}.
\end{align*}
Similarly, let $u\in \tilde K(S^4)_\mathcal{P}$ and $ch\,u=s\in H^4(S^4;\bm{Q})$ such that $s=b|_{S^4}$.
Assume that there exists the following homotopy commutative diagram:
\begin{align*}
\xymatrix{
S^4 \vee \bm{H}P^{n}_\mathcal{P} \ar[r]^-{\ell k\vee i} \ar[d]_{j} & \bm{H}P^\infty_\mathcal{P}\vee \bm{H}P^\infty_\mathcal{P} \ar[d]^{\nabla} \\
S^4 \times \bm{H}P^{n}_\mathcal{P} \ar[r]^-{f} & \bm{H}P^\infty_\mathcal{P}
}
\end{align*}
Then, $f^*b=ks\times 1+1\times b$ in $H^4(S^4 \times \bm{H}P^n_\mathcal{P};\bm{Q})$ and
\begin{align*}
f^*a=ku\times 1+1\times a+\sum_{i=1}^n\epsilon _i(k)u\times a^i
\end{align*}
in $\tilde K(S^4 \times \bm{H}P^n_\mathcal{P})_\mathcal{P}$, where $\epsilon _i(k) \in \bm{Z}_\mathcal{P}$.
We calculate $f^*ch\, a$ and $ch\, f^*a$ as follows:
\begin{align*}
f^*ch\,a=f^*\sum_{j=1}^\infty\frac{2b^j}{(2j)!}=\sum_{j=1}^\infty \frac{2}{(2j)!}(ks\times 1+1\times b)^j=ks\times 1+\sum_{j=1}^n \left( \frac{k}{(2j+1)!}s\times b^j+\frac{2}{(2j)!}1\times b^j\right),
\end{align*}
\begin{align*}
ch\, f^*a=ch\,\left( ku\times 1+1\times a+\sum_{i=1}^n\epsilon _i(k)u\times a^i \right)
=&ks\times 1+1\times \sum_{j=1}^n\frac{2}{(2j)!}b^j+\sum_{i=1}^n\sum_{j=1}^n\epsilon _i(k)s\times \left( \sum_{j=1}^n\frac{2}{(2j)!}b^j \right)^i\\
=&ks\times 1+\sum_{j=1}^n\frac{2}{(2j)!}1\times b^j+\sum_{i=1}^n\sum_{l=1}^n\sum_{j_1+\cdots +j_i=l}\frac{2^i\epsilon _i(k)}{(2j_1)!\cdots (2j_i)!}s\times b^l.
\end{align*}
Then we have the following formula:
\begin{align*}
\frac{k}{(2l+1)!}=\sum_{i=1}^n\sum_{\substack{j_1+\cdots +j_i=l\\ j_1, \cdots, j_i\geq 1}}\frac{2^i\epsilon _i(k)}{(2j_1)!\cdots (2j_i)!}.
\end{align*}
From this formula, one can see that there exists the number $\epsilon _i\in \bm{Q}$ such that $\epsilon _i(k)=\epsilon _ik$ for each $i$.
Of course, the sequence $\{\epsilon_i\}_{i=1}^\infty$ satisfy the following formula for each $l$:
\begin{align*}
\frac{1}{(2l+1)!}=\sum_{i=1}^n\sum_{\substack{j_1+\cdots +j_i=l\\ j_1, \cdots, j_i\geq 1}}\frac{2^i\epsilon _i}{(2j_1)!\cdots (2j_i)!}.
\end{align*}
For example, $\epsilon _1=1/6$, $\epsilon _2=-1/180$, $\epsilon _3=1/1512$ etc.
\begin{ex}
If $\aut P_k$ is trivial as a fibrewise $A_3$-space, then $7560$ divides $k$.
\end{ex}
\begin{proof}
This condition implies that $\epsilon_1k, \epsilon_2k, \epsilon_3k \in \bm{Z}$.
Thus $k$ is divided by $7560=2^33^35^17^1$.
\end{proof}
\begin{thm}\label{p-1}
Let $p$ be an odd prime.
The fibrewise topological group $\aut P_{k,\{p\}}$ is trivial as a fibrewise $A_{p-1}$-space if and only if $k \equiv 0\mod p^2$.
\end{thm}
\begin{proof}
Assume $\aut P_{k,\{p\}}$ is trivial as a fibrewise $A_{p-1}$-space for some $k\in \bm{Z}$.
Then $k\epsilon_{p-1}\in \bm{Z}_{(p)}$.
From this, to show $k\equiv 0 \mod p^2$, it is sufficient to show that $p\epsilon_{p-1}\not \in \bm{Z}_{(p)}$.
Obviously, $\epsilon _1,\cdots ,\epsilon _{(p-3)/2}\in \bm{Z}_{(p)}$.
Then
\begin{align*}
\epsilon _{(p-1)/2}=\frac{1}{p!}\mod \bm{Z}_{(p)}.
\end{align*}
From this, we have $p\epsilon _{(p+1)/2},\cdots ,p\epsilon _{p-3}\in \bm{Z}_{(p)}$.
Hence, by using the above formula for $l=p-1$, we obtain
\begin{align*}
0=\frac{2^{(p-1)/2}p\epsilon _{(p-1)/2}}{(p+1)!(2!)^{(p-3)/2}}\frac{p-1}{2}+p\epsilon _{p-1} \mod \bm{Z}_{(p)}.
\end{align*}
Therefore, we get
\begin{align*}
p\epsilon _{p-1}=-\frac{1}{(p+1)!(p-2)!}\mod \bm{Z}_{(p)}.
\end{align*}
This implies that $p\epsilon_{p-1}\not \in \bm{Z}_{(p)}$.
\par
Conversely, assume $k\equiv 0\mod p^2$.
Let $k'=k/p$.
Since $k'\equiv 0 \mod p$, we can extend $\nabla (\ell k'\vee i):S^4\vee \bm{H}P^{(p-1)/2}_{(p)}\to \bm{H}P^\infty_{(p)}$ over $S^4\times \bm{H}P^{(p-1)/2}_{(p)}$ by using Theorem \ref{p-1/2}.
Moreover, since $\pi_r(\bm{H}P^\infty_{(p)})=0$ for $2p+1<r<4p-2$, we can extend $\nabla (\ell k'\vee i):S^4\vee \bm{H}P^{p-2}_{(p)}\to \bm{H}P^\infty_{(p)}$ over $S^4\times \bm{H}P^{p-2}_{(p)}$.
Thus we can extend $\nabla (\ell k\vee i):S^4\vee \bm{H}P^{p-1}_{(p)}\to \bm{H}P^\infty_{(p)}$ over $S^4\times \bm{H}P^{p-1}_{(p)}$ since $\pi_{4p-1}(\bm{H}P^\infty_{(p)})\simeq \bm{Z}/p\bm{Z}$.
\end{proof}
\par

Now, we give the desired lower bound.
\begin{cor}
The number of the $A_n$-types of gauge groups of principal $\SU(2)$-bundles over $S^4$ is larger than $2^{\pi(2n+1)}$, where $\pi(m)$ represents the number of prime numbers less than or equal to $m$.
\end{cor}
\begin{proof}
Fix an odd prime $p$ and let $k=2^{i_2}3^{i_3}5^{i_5}\cdots p^{i_p}$ and $k'=2^{i_2'}3^{i'_3}5^{'i_5}\cdots p^{i'_p}$, where each $i_r$ or $i_r'$ equals to $0$ or $1$.
Then, from the result of \cite{Kon91}, Proposition \ref{trivialgaugegroup} , Proposition \ref{primetop} and Theorem \ref{p-1/2}, $\mathcal{G}(P_k)$ and $\mathcal{G}(P_{k'})$ are $A_{\frac{p-1}{2}}$-equivalent if and only if $k=k'$.
Therefore, there is at least $2^{\pi(p)}$ different types of the $A_n$-types of the gauge groups of principal $\SU(2)$-bundles over $S^4$ for $n\geq (p-1)/2$.
\end{proof}
\begin{rem}
Similarly, from Theorem \ref{p-1}, since each two of $\aut P_{0,\{p\}}$, $\aut P_{1,\{p\}}$ and $\aut P_{p,\{p\}}$ are not fibrewise $A_{p-1}$-equivalent, we have a sharper result: the number of the $A_n$-types of the gauge groups of principal $\SU(2)$-bundles over $S^4$ is larger than $2^{\pi(2n+1)-\pi(n+1)}3^{\pi(n+1)}$.
\par
From the prime number theorem, the number of $A_n$-types of the gauge groups of principal $\SU(2)$-bundles over $S^4$ has at least the growth of $2^{(2n+1)/\log(2n+1)}$.
\par
This corollary gives an alternative proof of Proposition \ref{counterexample} for $X=S^4$, but does not give the complete classification of the $A_\infty$-types of the gauge groups.
\end{rem}

   Mitsunobu Tsutaya,
   Department of Mathematics,
   Kyoto University\\
   E-mail address: tsutaya@math.kyoto-u.ac.jp
   
\end{document}